\documentclass[article,11pt]{amsart}
\usepackage[top=25mm,bottom=25mm,left=25mm,right=25mm]{geometry}
\usepackage{graphicx}

\newtheorem{lemma}{Lemma}

\newtheorem{proposition}{Proposition}
\newtheorem{theorem}{Theorem}
\newtheorem{corollary}{Corollary}

\newtheorem{example}{Example}

\DeclareMathOperator{\spann}{span}
\DeclareMathOperator{\diam}{diam}

\DeclareMathOperator{\loc}{loc}
\DeclareMathOperator{\comp}{c}
\DeclareMathOperator{\bound}{b}

\newcommand{\cD}{\mathcal{D}}
\newcommand{\cQ}{\mathcal{Q}}
\newcommand{\cA}{\mathcal{A}}

\newcommand{\cV}{\mathcal{V}}
\newcommand{\cK}{\mathcal{K}}

\newcommand{\cF}{\mathcal{F}}
\newcommand{\cE}{\mathcal{E}}

\newcommand{\R}{\mathbf{R}}
\newcommand{\1}{\mathbf{1}}

\newcommand{\pr}{\mathbf{P}}
\newcommand{\ex}{\mathbf{E}}

\begin{document}
\title[Spectral gap for fractional Schr\"odinger operators]{Spectral gap lower bound for the one-dimensional fractional Schr\"odinger operator in the interval}
\author{ Kamil Kaleta}
\address{Kamil Kaleta, Institute of Mathematics and Computer Science \\ Wroc{\l}aw University of Technology 
\\ Wyb. Wyspia{\'n}skiego 27, 50-370 Wroc{\l}aw, Poland}
\email{kamil.kaleta@pwr.wroc.pl}
\thanks{Research supported by the Polish Ministry of Science and Higher Education grant no. N N201 527338}

\begin{abstract}
We prove the uniform lower bound for the difference $\lambda_2 - \lambda_1$ between first two eigenvalues of the fractional Schr\"odinger operator, which is related to the Feynman-Kac semigroup of the symmetric $\alpha$-stable process killed upon leaving open interval $(a,b) \in \R $ with symmetric differentiable single-well potential $V$ in the interval $(a,b)$, $\alpha \in (1,2)$. "Uniform" means that the positive constant appearing in our estimate $\lambda_2 - \lambda_1 \geq C_{\alpha} (b-a)^{-\alpha}$ is independent of the potential $V$. In general case of $\alpha \in (0,2)$, we also find uniform lower bound for the difference $\lambda_{*} - \lambda_1$, where $\lambda_{*}$ denotes the smallest eigenvalue related to the antisymmetric eigenfunction $\varphi_{*}$. We discuss some properties of the corresponding ground state eigenfunction $\varphi_1$. In particular, we show that it is symmetric and unimodal in the interval $(a,b)$.
\end{abstract}

\maketitle

\section{Introduction and statement of results}

The main purpose of this paper is to prove an uniform lower bound for the spectral gap of the fractional Schr\"odinger operator with symmetric differentiable single-well potential on a bounded interval of the real line. Such an operator is related to the Feynman-Kac semigroup of the killed symmetric $\alpha$-stable process. To obtain this bound we study some basic properties of the first and second eigenfunction of this operator such as monotonicity and differentiability. Our work is motivated by the classical results obtained in \cite{bib:AB1, bib:AB2}, where the similar spectral problem was studied for the classical Schr\"odinger operator with the symmetric single-well potential on the interval.

Before we describe our results in details let us recall the basic definitions and facts. Let $(X_t)_{t \geq 0}$ be the symmetric $\alpha$-stable process of order $\alpha \in (0,2)$ in $\R$. This process is a Markov process with stationary independent increments and the characteristic function of the form $\ex^0[\exp{(i \xi X_t)} = \exp (-t|\xi|^{\alpha})$, $\xi \in \R$, $t>0$. As usual, $\ex^x$ denotes the expected value of the process starting at $x \in \R$. Let $(a,b) \subset \R$, $-\infty<a<b<\infty$, be an open interval and let $\tau_{(a,b)} = \inf\left\{ t \geq 0: X_t \notin (a,b)\right\}$ be the first exit time of $X_t$ from $(a,b)$. 

The Feynman-Kac semigroup $(T_t)_{t \geq 0}$ for the symmetric $\alpha$-stable process $X_t$ killed upon leaving $(a,b)$ and for potential $V \in \cV^{\alpha}((a,b))$ is defined as
\begin{align}
\label{def:FKS}
T_t f(x) = \ex^x\left[ \exp\left(-\int_0^t V(X_s)ds\right)f(X_t); \tau_{(a,b)} >t\right], \ \ \ \ f \in L^2((a,b)), \ t>0, \ x \in (a,b),
\end{align}
where $\cV^{\alpha}((a,b))$ is a class of functions $V:(a,b) \rightarrow \R$ specified by the following three conditions:
\begin{itemize}
\item[(i)] integrability: $V$ extended to $\R$ by putting $0$ outside $(a,b)$ is in the Kato class $\cK^{\alpha}$ for the symmetric $\alpha$-stable process $X_t$.
Formal definition of $\cK^{\alpha}$ is given in Section 2.
\item[(ii)] symmetry: $V(x) = V(b+a-x)$ for $x \in (a,b)$.
\item[(iii)] differentiability and monotonicity: $V^{'}$ exists in $(a,b)$ and $V^{'}(x) \leq 0$ for $x \in (a,(a+b)/2)$.
\end{itemize}
\noindent
In the above definition we use the convention that potentials $V$ are defined on the interval $(a,b)$. However, very often, it will be useful to see the potential $V$ as a function extended to whole real line $\R$ by putting $0$ outside $(a,b)$. Notice also that the assumption (i) is an integrability condition under which the above Feynman-Kac semigroup is well defined (see \cite{bib:BB1, bib:BB2}).
Moreover, it immediately follows from the assumptions (ii) and (iii) that $V$ is a symmetric function, which is continuous, bounded from below, nonincreasing in $(a,(a+b)/2)$ and nondecreasing in $((a+b)/2,b)$. In the terminology of \cite{bib:AB1} potentials with a such monotonicity property are called symmetric single-well. We refer to the potentials from the class $\cV^{\alpha}((a,b))$ as the \emph{symmetric differentiable single-well potentials} on the interval $(a,b)$.

The operators $T_t$ are symmetric and form a strongly continuous semigroup on $L^2((a,b))$. The infinitesimal generator $L$ of the semigroup $(T_t)_{t \geq 0}$ is defined formally by
$$
L f = \lim_{t \downarrow 0} \frac{T_tf - f}{t}
$$
for such $f \in L^2((a,b))$ for which this limit exists in $L^2((a,b))$. The set of all such functions is called the domain of $L$ and is denoted by $\cD(L)$. Similarly, we define $Lf(x) = \lim_{t \downarrow 0} (T_t f(x) - f(x))/t$ for any $f \in C((a,b))$ and $x \in (a,b)$ for which the limit exists. It is easy to show that if $f \in C^{\infty}_{\comp}((a,b))$, then $L f(x)$ is well defined and 
$$
L f(x) = -(-\Delta)^{\alpha/2}f(x) - V(x)f(x), \ \ \ \ \ \ \ \ \ x \in (a,b),
$$ 
where $-(-\Delta)^{\alpha/2}$ is the fractional Laplacian of order $\alpha$ (see \cite[p. 11-13]{bib:BBKRSV}). Note that the operator $(-\Delta)^{\alpha/2}$ will not be used in any essential way in proving our eigenvalue gap estimates. We just want to emphasize the connection between the operators $L$ and $(-\Delta)^{\alpha/2}$. By this correspondence, we refer to the operator $-L$ as the fractional Schr\"odinger operator with the potential $V$ on interval $(a,b)$. 

In recent years Schr{\"o}dinger operators based on non-local pseudodifferential operators have been intensively studied. One of the most well known result is the so-called Hardy-Lieb-Thirring inequality obtained in 2008 by R. Frank, E. Lieb and R. Seiringer \cite{bib:FLS1}, which is connected with the problem of the stability of matter \cite{bib:LS}. In the last 30 years many results concerning the fractional Schr{\"o}dinger operators and relativistic Schr{\"o}dinger operators have been obtained \cite{bib:CMS, bib:HIL, bib:Z, bib:CS, bib:CS1, bib:CS2, bib:BB1, bib:BB2, bib:KS, bib:KaKu, bib:KL, bib:LMa}. These results are about functional integration, structure of spectrum, conditional gauge theorem, estimates of eigenfunctions and intrinsic ultracontractivity. Most of these results are obtained by using the probabilistic and potential theoretic methods. 

The fact that the interval $(a,b)$ is bounded implies that for any $t>0$ the operator $T_t$ maps $L^2((a,b))$ into $L^{\infty}((a,b))$. It follows from the general theory of semigroups that there exists an orthonormal basis of egenfunctions $\left\{\varphi_n\right\}$ in $L^2((a,b))$ and corresponding sequence of eigenvalues 
$$
\lambda_1 < \lambda_2 \leq \lambda_3 \leq ... \to \infty
$$
satisfying
\begin{align*}
T_t \ \varphi_n & = e^{- \lambda_n t} \ \varphi_n, \\
L \varphi_n & = - \lambda_n \ \varphi_n, \ \ \ \ \ \ \ \ \ n \geq 1.
\end{align*}
We may and do choose the basis $\left\{\varphi_n\right\}$ so that $\varphi_n$ is either symmetric (i.e., $\varphi_n(x)= \varphi_n(a+b-x)$ for $x \in (a,b)$) or antisymmetric (i.e., $\varphi_n(x)= - \varphi_n(a+b-x)$ for $x \in (a,b)$). Moreover, each eigenfunction $\varphi_n$ is continuous and  bounded and all $\lambda_n$ have finite multiplicities. Additionally, $\lambda_1$ is simple and the corresponding eigenfunction, the so-called ground state eigenfunction, can be assumed to be strictly positive on $(a,b)$. Our main concern in this paper is the difference $\lambda_2 - \lambda_1 >0$, which is called the spectral gap. All above defined objects depend on the stability parameter $\alpha \in (0,2)$, the interval $(a,b)$ and the potential $V \in \cV^{\alpha}((a,b))$. However, for more simplicity, we prefer to omit this dependence in our notation. 

Let us point out that $(-\Delta)^{\alpha/2} \varphi_n(x)$ is well defined for all $x \in (a,b)$, $n \geq 1$, and we have 
$$
-L \varphi_n (x) = (-\Delta)^{\alpha/2}\varphi_n(x) + V(x) \varphi_n(x)  =  \lambda_n \varphi_n(x), \ \ \ \ \ \ \ x \in (a,b).
$$
By the fact that each $\varphi_n$ may be extended to whole $\R$ by putting $0$ outside $(a,b)$, discussed in this paper spectral problem can be seen as the eigenproblem for the fractional Schr\"odinger operator on the interval $(a,b)$ with Dirichlet exterior conditions (that is, outside the interval $(a,b)$). 

Mentioned above spectral problem has been widely studied for classical Schr\"odinger operators $-\Delta + V$ acting on $L^2(D)$ with Dirichlet boundary conditions, where $D$ is a bounded domain in $\R^d$, $d \geq 1$. Motivated by problems in mathematical physics concerning the behaviuor of free Boson gases, M. van den Berg \cite{bib:Be} made the following conjecture. If $D \subset \R^d$ is convex such that $\diam(D) < \infty$ and $V$ is nonnegative convex potential in $D$, then 
\begin{align}
\label{eq:clsp}
\lambda_{2,D}^V - \lambda_{1,D}^V > \frac{3 \pi^2}{\diam(D)^2},
\end{align}
where $\lambda_{1,D}^V$ and $\lambda_{2,D}^V$ are the first and the second eigenvalue of $-\Delta + V$ acting on $L^2(D)$ with Dirichlet boundary conditions. This problem has been widely studied by many authors \cite{bib:SWYY, bib:YZ, bib:Da2, bib:BM, bib:Sm, bib:Li, bib:AC}. In particular, the strict inequality \eqref{eq:clsp} was obtained in 2010 by B. Andrews and J. Clutterbuck \cite{bib:AC}. Let us point out that the this conjecture for intervals on the real line and for arbitrary nonnegative convex potentials was proved earlier by R. Lavine \cite{bib:Lav}.

The classical result which is the most related to our one was obtained by M. Ashbaugh and R. Benguria \cite{bib:AB1,bib:AB2}. They studied this problem in one dimension when a domain $D$ is just a bounded interval and showed the inequality \eqref{eq:clsp} for the different class of \emph{symmetric single-well potentials} $V$ that are integrable in $D \subset \R$. This class includes the symmetric convex potentials, as well as a variety of nonconvex (but symmetric) potentials. 

The problem of eigenvalue estimates and the spectral gap lower bound has also been studied for fractional Laplacian $-(-\Delta)^{\alpha/2}$ (i.e. $V \equiv 0$) on bounded domains of $\R^d$ with Dirichlet exterior conditions \cite{bib:CS4, bib:DBl, bib:BK2, bib:BK3, bib:BK4, bib:DK, bib:Kw, bib:DBlMH}. In one dimension (when $D$ is an interval) eigenvalue gaps estimates follow from results in \cite{bib:BK1} ($\alpha = 1$) and \cite{bib:CS3} ($\alpha > 1$). Moreover, the recent papers \cite{bib:KKMS} ($\alpha = 1$), \cite{bib:Kw1} ($\alpha \in (0,2)$) contain new asymptotic formulas for eigenvalues, which can be used to find the numerical bounds for eigenvalue gaps.

Now we formulate the main results of this chapter. The following variational formula for the gaps between eigenvalues of $L$ is rather standard fact. In particular, this formula gives an integral representation for the spectral gap $\lambda_2 - \lambda_1$ and will be a starting point of our proofs. In fact, it is a fractional extension of the classical variational formula for the eigenvalue gaps of the classical Schr\"odinger operators which can be found for example in \cite{bib:Sm}. For the version of a such formula for the fractional Laplacian (i.e. $V \equiv 0$) we refer to \cite{bib:DK}. Denote by $L^2((a,b),\varphi_1^2)$ the space of square-integrable functions on the interval $(a,b)$ with measure $\varphi_1^2(x)dx$.

\begin{proposition}
\label{prop:var}
Assume that $\alpha \in (0,2)$. Let $V \in \cV^{\alpha}((a ,b))$, $-\infty<a<b<\infty$. Then for every $n \geq 2$ we have
\begin{align}
\label{eq:var}
\lambda_n - \lambda_1 = \inf_{f \in \cF_n \cap C_{\bound}((a,b))} \frac{\cA_{-\alpha}}{2} \int_{a}^b \int_{a}^b \frac{(f(x)-f(y))^2}{|x-y|^{1+\alpha}} \varphi_1(x) \varphi_1(y) dx dy, 
\end{align}
where
$$
\cF_n = \left\{f \in L^2((a,b),\varphi_1^2): \int_a^b f^2(x) \varphi_1^2(x)dx = 1, \int_a^b f(x) \varphi_1(x) \varphi_i(x) dx = 0, \ 1 \leq i\leq n-1\right\}
$$
and
\begin{align}
\label{eq:constant}
\cA_{\gamma} = \frac{\Gamma((1-\gamma)/2)}{2^{\gamma}\sqrt{\pi}|\Gamma(\gamma/2)|} .
\end{align}
\noindent
Moreover, the infimum in \eqref{eq:var} is achieved for $f = \varphi_n / \varphi_1$. 
\end{proposition}
\noindent
The proof of the above proposition is given in Section 3. Let us also point out that Proposition \ref{prop:var} is a consequence of the standard variational formula for eigenvalues and a special case of \cite[Theorems 2.6 and 2.8]{bib:CFTYZ}.

Our first main results are the following theorems concerning the properties of eigenfunctions. The first one is about the monotonicity of $\varphi_1$. We extend some earlier ideas from \cite{bib:BK3, bib:BKM} and \cite{bib:BM}, and prove that the ground state eigenfunction is unimodal in the interval $(a,b)$. This property will be our main argument in proving eigenvalue gap estimates.

\begin{theorem}
\label{th:gsp}
Let $\alpha \in (0,2)$. Let $V \in \cV^{\alpha}((a,b))$, $-\infty<a<b<\infty$. Then $\varphi_1$ is symmetric and unimodal in $(a,b)$, i.e., $\varphi_1$ is nondecreasing in $(a,(a+b)/2)$ and nonincreasing in $((a+b)/2,b)$.
\end{theorem}

\begin{theorem}
\label{th:gsp1}
Let $\alpha \in (1,2)$. Let $V \in \cV^{\alpha}((a,b))$, $-\infty<a<b<\infty$. Then for all $n \geq 1$ the derivative of $\varphi_n$ exists in $(a,b)$. Moreover, if $[c,d] \subset (a,b)$, then there exists a constant $C_{V,n,\alpha,a,b,c,d}$ such that for all $x \in [c,d]$ we have
$$
\left|\frac{d}{dx} \varphi_n(x)\right| \leq C_{V,n,\alpha,a,b,c,d} \left\|\varphi_n\right\|_{\infty}.
$$
\end{theorem}

Let us recall that our orthonormal basis $\left\{\varphi_n\right\}$ is chosen so that each $\varphi_n$ is either symmetric or antisymmetric in $(a,b)$. It follows from the fact that $\left\{\varphi_n\right\}$ is an orthonormal basis that among $\varphi_n$ there are infinitely many antisymmetric functions in $(a,b)$. Denote by $\lambda_{*}$ the smallest eigenvalue corresponding to the antisymmetric eigenfunction $\varphi_{*}$. It is a natural hypothesis that $\lambda_{*} = \lambda_2$. For the classical Schr\"odinger operator on the the interval this fact is well known. It is a consequence of the Courant-Hilbert theorem which states that $\varphi_2$ has exactly two nodal domains (the interval consists of exactly two subintervals on which the sign of $\varphi_2$ is fixed). In our case this problem is more complicated. This is due to the fact that any version of the Courant-Hilbert theorem is not known for operators which are non-local. Despite this fact, this hypothesis was proved by R. Ba\~nuelos and T. Kulczycki \cite{bib:BK2} for $\alpha = 1$ and $V \equiv 0$. Very recently, M. Kwa\'snicki proved this hypothesis in \cite{bib:Kw1} for $\alpha \in (1,2)$ and $V \equiv 0$ also.
                                                                                                       
The following lower bound for $\lambda_{*} - \lambda_1$ is derived by using Theorem \ref{th:gsp} only. 

\begin{theorem}
\label{th:sga}
Assume that $\alpha \in (0,2)$. Let $V \in \cV^{\alpha}((a,b))$, $-\infty<a<b<\infty$. Let $\lambda_{*}$ be the smallest eigenvalue corresponding to the antisymmetric eigenfunction $\varphi_{*}$. Then we have
\begin{align}
\label{eq:sga}
\lambda_{*} - \lambda_1 \geq \frac{\cA_{-\alpha}}{(b-a)^{\alpha}}.
\end{align}
\end{theorem}

The next theorem is the main result of this paper.

\begin{theorem}
\label{th:sg}
Assume that $\alpha \in (1,2)$. Let $V \in \cV^{\alpha}((a,b))$, $-\infty<a<b<\infty$. Then we have
\begin{align}
\label{eq:sg}
\lambda_2 - \lambda_1 \geq \frac{C^{(3)}_{\alpha}}{(b-a)^{\alpha}},
\end{align}
where $C^{(3)}_{\alpha} = \frac{\cA_{-\alpha}}{4}\left(\frac{1}{9}\right)^{\frac{\alpha+1}{\alpha-1}}$.
\end{theorem}
\noindent

The following theorem contains the new inequality, which is a crucial tool used in proving bound \eqref{eq:sg}. This inequality may be of independent interest.  

\begin{theorem}
\label{th:fpi}
Assume that $\alpha \in (1,2)$. Let $-\infty<a<b<\infty$ and let $f$ be a Lipschitz function in $[a,b]$. If $f(a) = 0$, then 
\begin{align}
\label{eq:fpi}
\int_a^b \int_a^b \frac{(f(x)-f(y))^2}{|x-y|^{1+\alpha}}dxdy \geq C^{(4)}_{\alpha}\frac{f^2(b)}{(b-a)^{\alpha-1}} ,
\end{align}  
where $C^{(4)}_{\alpha} = \left(\frac{1}{9}\right)^{\frac{\alpha+1}{\alpha-1}}$. Similarly, if $f(b)=0$, then 
\begin{align*}
\int_a^b \int_a^b \frac{(f(x)-f(y))^2}{|x-y|^{1+\alpha}}dxdy \geq C^{(4)}_{\alpha}\frac{f^2(a)}{(b-a)^{\alpha-1}}.
\end{align*}
\end{theorem}
\noindent
Let us point out that the above result is analogous to the following fact from the elementary calculus. If $f \in C^1([a,b])$ and $f(a) = 0$, then the inequality
$$
\int_a^b (f^{\prime}(y))^2 dy \geq \frac{1}{b-a}\left(\int_a^b f^{\prime}(y) dy\right)^2 = \frac{1}{b-a} f^2(b)
$$
holds.

The following counterexample shows that in the case $\alpha \in (0,1)$ the inequality \eqref{eq:fpi} does not hold with any positive constant. 
\begin{example}
\label{ex:counter}
Assume that $\alpha \in (0,1)$. Let $f$ be a $C^{\infty}$-class function such that 
\begin{align}
\label{eq:counterex}
f(x) = 
\left\{
\begin{array}{ll} 0, & x <1/4,\\
\in [0,1),           & 1/4 \leq x < 1/2, \\
1,                   & x \geq 1/2.
\end{array} \right. 
\end{align}
Consider the sequence of functions of the form $f_n(x) = f(nx)$, $n \geq 1$. Clearly, each $f_n$ is a $C^{\infty}$-class function such that $f_n(0) = 0$, $f_n(1)=1$. However,
\begin{align}
\label{eq:counter}
\int_0^1 \int_0^1 \frac{(f_n(x)-f_n(y))^2}{|x-y|^{1+\alpha}}dxdy \to 0 \ \ \ \ \ \ \ \text{as} \ \ n \to \infty.
\end{align}
\end{example}

The justification of the above example is a very special version of similar one in \cite[Section 2]{bib:D} and is given in Section \ref{sec:crucineq}.

Note that the constants $C^{(3)}_{\alpha}$ and $C^{(4)}_{\alpha}$ obtained in Theorem \ref{th:sg} and Theorem \ref{th:fpi}, respectively, are not optimal. As we will see below, the inequality \eqref{eq:fpi} is an important argument used in proving the bound \eqref{eq:sg}. Indeed, we have $C^{(3)}_{\alpha} = \cA_{-\alpha} / 4 \ C^{(4)}_{\alpha}$. It follows that by improving the constant $C^{(4)}_{\alpha}$ in the inequality \eqref{eq:fpi}, one can improve the constant $C^{(3)}_{\alpha}$ in \eqref{eq:sg}. Notice also that in view of Theorem \ref{th:sga} another way to improve the constant in Theorem \ref{th:sg} is to show that $\lambda_{*} = \lambda_2$.

A consequence of Theorem \ref{th:fpi} is the following fractional version of the weighted Poincar\'e inequality.

\begin{corollary} 
\label{cor:fwpi}
Assume that $\alpha \in (1,2)$. Let $-\infty<a<b<\infty$ and let $g:[a,b] \rightarrow \R$ be continuous, nonincreasing and strictly positive in $[a,b)$. Let $f$ be a Lipschitz function on the interval $[a,b]$ such that $f(a) = 0$. Then 
\begin{align}
\label{eq:fwpi}
\int_a^b \int_a^b \frac{(f(x)-f(y))^2}{|x-y|^{1+\alpha}}g(x)g(y) dxdy \geq \frac{C^{(4)}_{\alpha}}{(b-a)^{\alpha}}\int_a^b f^2(x)g^2(x) dx,
\end{align}  
where $C^{(4)}_{\alpha}$ is given in Theorem \ref{th:fpi}. 
\end{corollary}

The paper is organized as follows. In Section 2 we introduce additional notation and collect various facts which are used below. Section 3 contains the proof of the variational formulas for eigenvalue gaps. In Section 4 we prove the properties of eigenfunctions. Section 5 contains the proof of Theorem \ref{th:sga}. In Section 6 we show the crucial inequality in Theorem \ref{th:fpi} and justify Example \ref{ex:counter}. Section 7 contains the proof of main theorem concerning the spectral gap lower bound. 

\section{Preliminaries}

Let $\alpha \in (0,2)$. By $C_{\alpha,\kappa}$ we always mean a strictly positive and finite constant depending on $\alpha$ and parameter $\kappa$. We adopt the convention that constants in some proofs may change their value from one use to another. However, very often, especially in the statements of our results, we write $C_\kappa^{(1)}, C_\kappa^{(2)}$ etc to distinguish between constants.

We now summarize the properties of the symmetric $\alpha$-stable process and some facts from its potential
theory. For further information on the potential theory of stable processes we refer to
\cite{bib:Bo, bib:Ku1, bib:CS1, bib:BBKRSV, bib:BKK}.

Let $X = (X_t)_{t \geq 0}$ be the standard one-dimensional symmetric $\alpha$-stable process with the L\'evy measure $\nu(dx)=\cA_{-\alpha}|x|^{-1-\alpha} dx$, where the constant $\cA_{-\alpha}$ is given by \eqref{eq:constant}. By $\pr^x$ we denote the distribution of the process starting at $x \in \R$. For each fixed $t>0$ the transition density $p(t,y-x)$, $t> 0$ , $x,y \in \R$, of the process $X$ is a continuous and bounded function on $\R \times \R$ satisfying the following estimates
\begin{align}
\label{eq:kernest}
C_{\alpha}^{-1}\left(\frac{t}{|y-x|^{1+\alpha}}\wedge t^{-1/\alpha}\right) & \leq p(t,y-x)
\leq C_{\alpha} \left(\frac{t}{|y-x|^{1+\alpha}}\wedge t^{-1/\alpha} \right).
\end{align}
It is known that for any $x, y \in D$, $x \neq y$, we have 
\begin{align}
\label{eq:levy1}
\lim_{t \to 0^{+}} \frac{p(t,y-x)}{t} = \frac{\cA_{-\alpha}}{|x-y|^{1+\alpha}}.
\end{align}
Denote by 
\begin{align}\label{eq:BM}
q(t,y-x) = \frac{1}{\sqrt{4 \pi t}} \exp\left(-\frac{|x-y|^2}{4t}\right), \ \ \ t>0, \ x,y \in \R,
\end{align}
the transition density of the standard one dimensional Brownian motion $(B_t)_{t \geq 0}$ runing at twice the speed. 

It is well known that the symmetric $\alpha$-stable process $(X_t)_{t \geq 0}$ can be represented as 
$$
X_t = B_{\eta_t},
$$
where $(\eta_t)_{t \geq 0}$ is an $\alpha/2$-stable subordinator independent of $(B_t)_{t \geq 0}$ (see \cite{bib:BBKRSV}). Thus
\begin{align}
\label{eq:subfor}
p(t,y-x) = \int_0^\infty q(s,y-x)g_{\alpha/2}(t,s)ds,
\end{align}
where $g_{\alpha/2}(t,s)$ is the transition density of $\eta_t$. 

It is known that when $\alpha < 1$, the process $X$ is transient with potential kernel
\cite{bib:BlG, bib:La}
$$
K^{(\alpha)}(y-x) = \int_0^\infty p(t,y-x)dt = \cA_{\alpha} |y-x|^{\alpha-1}, \quad x,y \in \R.
$$
Whenever $\alpha \geq 1$ the process is recurrent (pointwise recurrent when $\alpha > 1$). In this case
we can consider the compensated kernel \cite{bib:BlGR}, that is, for $\alpha \geq 1$ we put
$$
K^{(\alpha)}(y-x) = \int_0^\infty \left( p(t,y-x ) - p(t,x_0)\right)dt,
$$
where $x_0 = 0$ for $\alpha > 1$, and $x_0 = 1$ for $\alpha = 1$. In this case
$$
K^{(\alpha)} (x) = \frac{1}{\pi} \log\frac{1}{|x|}
$$
for $\alpha = 1$ and
$$
K^{(\alpha)} (x) = (2\Gamma(\alpha) \cos(\pi\alpha/2))^{-1} |x|^{\alpha-1}, \ \ \ \ x \in \R,
$$
for $\alpha > 1$. Note that $K^{(\alpha)} (x) \leq 0$ if $\alpha > 1$.

We say that the Borel function $V: \R \to \R$ belongs to the Kato class $\cK^\alpha$ corresponding to the symmetric $\alpha$-stable process $X$
if $V$ satisfies either of the two equivalent conditions (see \cite{bib:Z} and \cite[(2.5)]{bib:BB2})
$$
\lim_{\epsilon \rightarrow 0} \sup_{x \in \R} \int_{|y-x|<\epsilon} |V(y)| |K^{(\alpha)} (y-x)| dy = 0,
$$
$$
\lim_{t \rightarrow 0} \sup_{x \in \R} \ex^x \left[\int_0^t |V(X_s)| ds\right] = 0.
$$
For instance, if $V(x) = (1-x^2)^{-\beta}$, $\beta>0$, then $V \in \cK^\alpha$ and $V \in \cV^{\alpha}((-1,1))$ provided that $\beta < \alpha \wedge 1$. It can be verified directly that for every $\alpha \in (0,2)$, $\cK^\alpha \subset L^1_{\loc}(\R)$. 

We denote by $p_D(t,x,y)$ the transition density of the process killed upon exiting an open bounded set $D \subset \R$.
It satisfies the relation
\begin{align*}
p_D(t,x,y) & = p(t,y-x) - \ex^x[p(t-\tau_D,y-X_{\tau_D}); \tau_D \leq t], \quad x, y \in D, \, t > 0,
\end{align*}
where $\tau_D = \inf\{t > 0: \, X_t \notin D\}$ is the first exit time from $D$. For every $t > 0$, $x,y \in D$, we have
\begin{align}
\label{eq:kernest1}
0 < p_D(t,x,y) \leq p(t,y-x).
\end{align}

For any $x, y \in D$, $x \neq y$, we have 
\begin{align}
\label{eq:levy2}
\lim_{t \to 0^{+}} \frac{p_D(t,x,y)}{t} = \frac{\cA_{-\alpha}}{|x-y|^{1+\alpha}}.
\end{align}
\noindent
In view of \eqref{eq:levy1}, this is a consequence of the fact that for every $x,y \in D$ we have 
$$
\lim_{t \to 0^{+}} \frac{1}{t} \ex^x[p(t-\tau_D,y-X_{\tau_D}); \tau_D \leq t]= 0.
$$
Indeed, for $x,y \in D$, $t> 0$, we obtain
\begin{align}
\label{eq:killing2}
\frac{1}{t} \ex^x[p(t-\tau_D,y-X_{\tau_D}); \tau_D \leq t] \leq \frac{1}{t} \ex^x\left[\frac{C t}{|y-X_{\tau_{D}}|^{1+ \alpha}}; \tau_{D} \leq t\right] \leq C \frac{\pr^x(\tau_{D} \leq t)}{\delta(y)^{1+\alpha}},
\end{align}
where $\delta(y) = \inf\left\{|z-y|: z \in \partial D\right\}$. 

Let us recall that the semigroup of the process killed upon leaving set $D$ is given by 
$$
P^D_t f(x) = \ex^x[f(X_t); \tau_D > t] = \int_D f(y) p_D(t,x,y) dy, \ \ \ \ \ \ \ \ \ f \in L^2(D), \ x \in D.
$$
The Green operator of an open bounded set $D$ is denoted by $G_D$. We set 
$$
G_D(x,y) = \int_0^\infty p_D(t,x,y)dt
$$
and call $G_D(x,y)$ the Green function for $D$. We have
$$
G_D f(x) = \ex^x\left[\int_0^{\tau_D} f(X_t) dt \right] = \int_D G_D(x,y)f(y)dy,
$$
for non-negative Borel function $f$ on $\R$. 

We now discuss some selected properties of the Feynman-Kac semigroup for the fractional Schr\"odinger operator with a potential $V$ on bounded intervals of $\R$, which are needed below. For the rest of this section we assume that $D$ is a bounded interval $(a,b) \subset \R$, $a<b$, and $V \in \cV^{\alpha}((a,b))$. 
We refer the reader to \cite{bib:BB1, bib:BB2, bib:CS, bib:CS1, bib:CZ} for more systematic treatment of fractional Schr\"odinger operators.

The $V$-Green operator for $(a,b)$ is defined by
\begin{align*}
G^V_{(a,b)} f(x)  = \int_0^\infty T_t f(x) dt =
\ex^x \left[\int_0^{\tau_{(a,b)}} e^{-\int_0^t V(X_s) ds} f(X_t) dt \right],
\end{align*}
for non-negative Borel functions $f$ on $(a,b)$. The corresponding gauge function is given by (see e.g. \cite[p. 58]{bib:BB1}, \cite{bib:CS1,bib:CZ})
$$
u_{(a,b)}(x) = \ex^x\left[e^{-\int_0^{\tau_{(a,b)}} V(X_s) ds}\right], \ \ \ \ \  x \in (a,b).
$$  
When it is bounded in $(a,b)$, then $((a,b),V)$ is said to be gaugeable. It is easy to check that if $V \geq 0$ on $(a,b)$, then gaugebility holds. 

The following perturbation type formula will be an important argument in the proof of differentiability of eigenfunctions. For the potential $V$ such that $((a,b),V)$ is gaugable and for bounded function $f$ we have (see \cite[Formula 9]{bib:BB1}) 
\begin{align}
\label{eq:pertform}
G^V_{(a,b)} f(x) = G_{(a,b)} f(x) + G_{(a,b)} (V G^V_{(a,b)} f)(x), \ \ \ \ \ x \in (a,b).
\end{align}
Since $V \in \cV^{\alpha}((a,b))$, $\inf V > - \infty$. Observe that when $V \equiv 0$, then $(T_t)_{t \geq 0}$ is an usual semigroup of the symmetric stable process killed upon leaving $(a,b)$. That is for each $t \geq 0$ we have $T_t = P_t^{(a,b)}$. Generally, for $f \geq 0$ we have
$$
T_t f (x) \leq e^{-\inf V t} P^{(a,b)}_t f(x), \ \ \ \ \ \ t> 0, \ \ x \in (a,b).
$$
Using this fact, \eqref{eq:kernest}, \eqref{eq:kernest1}, and the Riesz-Thorin interpolation theorem, it can be shown that for each $t>0$ the operators $T_t:L^p((a,b)) \rightarrow L^q((a,b))$, $1 \leq p \leq q \leq \infty$, are bounded. Denote by $\left\|T_t\right\|_{p,q}$ the norm of a such operators.

Recall that the class $\cV^{\alpha}((a,b))$ contains the signed potentials $V$. So we do not exclude the case that the operators $T_t$ are not sub-Markovian. However, each operator $T_t$ can be transformed to be sub-Markovian by adding a constant to the potential $V$. Indeed, if $\inf V < 0$, then we put $V_0 = V - \inf V$. Clearly, $V_0 \geq 0$, and thus, the corresponding Feynman-Kac semigroup is sub-Markovian. Denote this semigroup by $(T^0_t)_{t \geq 0}$. It can be checked directly that its generator is the operator $L_0=L+\inf V$ with purely discrete spectrum of the form $\left\{-\lambda_1 + \inf V,-\lambda_2 + \inf V, -\lambda_3 + \inf V,... \right\}$. However, both operators $L$ and $L_0$ have the same eigenvalue gaps $\lambda_n - \lambda_1$, $n \geq 1$. We will use this translation invariance property of the eigenvalue gaps in the sequel. 

\section{Variational formula}

We now justify the variational formula in Proposition \ref{prop:var}. In fact, Proposition \ref{prop:var} is a consequence of the standard variational formula for eigenvalues and a special case of \cite[Theorems 2.6 and 2.8]{bib:CFTYZ}. By using the above remark on the translation invariance of the eigenvalue gaps, we may and do assume that $V \geq 0$ and that the corresponding Feynman-Kac semigroup $(T_t)_{t \geq 0}$ is sub-Markovian. For every $t>0$ define operators:
$$
\widetilde T_t f = e^{\lambda_1 t}\varphi_1^{-1}  T_t (\varphi_1 f), \ \ \ \ \ \ \ \ f \in L^2((a,b),\varphi_1^2).
$$
It is easy to see that operators $\widetilde T_t$, $t>0$, form a semigroup of symmetric Markov operators on $L^2((a,b),\varphi_1^2)$ such that 
$$
\widetilde T_t \left(\frac{\varphi_i}{\varphi_1}\right) = e^{-(\lambda_i - \lambda_1)t} \frac{\varphi_i}{\varphi_1}, \ \ \ \ \ \ n \geq 1.
$$
Let
\begin{align}
\label{eq:var1}
\widetilde \cE(f,f) = \lim_{t \to 0^{+}} \frac{1}{t} (f - \widetilde T_t f, f)_{L^2((a,b), \varphi_1^2)}
\end{align}
for $f \in L^2((a,b),\varphi_1^2)$. It is known that the form $\widetilde \cE$ with its natural domain 
$$\cD(\widetilde \cE) = \left\{f \in L^2((a,b), \varphi_1^2): \widetilde \cE(f,f) < \infty\right\}$$
is the Dirichlet form corresponding to the semigroup $(\widetilde T_t)_{t>0}$ \cite[p. 23]{bib:FOT}. By the standard variational formula for eigenvalues we have 
$$
\lambda_n - \lambda_1 = \inf_{f \in \cF_n} \widetilde \cE(f,f), \ \ \ \ \ \ \ \ \ n \geq 2,
$$
and the infimum is achieved for $f = \varphi_n / \varphi_1$. Note that by the intrinsic ultracontractivity of the Feynman-Kac semigroup $(T_t)_{t \geq 0}$ (see \cite[Theorem 14 and (2)]{bib:Ku2}), each function $f =\frac{\varphi_n}{\varphi_1}$, $n \geq 2$, is bounded and continuous in the interval $(a,b)$. Thus to complete the proof of Proposition \ref{prop:var} it is enough to show that for all $f \in C_{\bound}((a,b))$ we have 
\begin{align}
\label{eq:stand}
\widetilde \cE(f,f) = \frac{\cA_{-\alpha}}{2} \int_{a}^b \int_{a}^b \frac{(f(x)-f(y))^2}{|x-y|^{1+\alpha}} \varphi_1(x) \varphi_1(y) dx dy.
\end{align}
The equality \eqref{eq:stand} is a special case of more general results in \cite[Theorems 2.6 and 2.8]{bib:CFTYZ}. However, its direct proof in our case is very simple and we decided to provide it. 

The proof of the equality \eqref{eq:stand} consists of the two following lemmas.  

\begin{lemma}
\label{lm:FKlevy}
Assume that $\alpha \in (0,2)$. Let $-\infty<a<b<\infty$ and $V \in \cV^{\alpha}((a,b))$ be a nonnegative potential. Let $F:(a,b) \times (a,b) \rightarrow \R$ be a bounded, Borel measurable function such that for each $x \in (a,b)$ $F(x,y) \to 0$ as $y \to x$. Denote
$$
g_t(x) = \ex^x \left[\frac{e^{-\int_0^t V(X_s)ds} - 1}{t} F(x,X_t) ; \tau_{(a,b)} > t\right].
$$
Then we have $\left\|g_t \varphi_1\right\|_{L^1((a,b))} \to 0$ as $t \to 0^{+}$. 
\end{lemma} 

\begin{proof}
First we show that $g_t(x) \to 0$ as $t \to 0^{+}$ pointwise, for any $x \in (a,b)$. Fix $x \in (a,b)$. Let $\delta(x) = (b-x) \wedge (x-a)$. Denote $\tau_x = \tau_{(x-\delta(x)/2, x+ \delta(x)/2)}$. We have 
\begin{align*}
g_t(x) & = \ex^x \left[\frac{e^{-\int_0^t V(X_s)ds} - 1}{t} F(x,X_t) ; \tau_x > t\right] \\ & \ \ \ + \ex^x \left[\frac{e^{-\int_0^t V(X_s)ds} - 1}{t} F(x,X_t) ; \tau_{(a,b)} > t \geq \tau_x \right] = g^{(1)}_t(x) + g^{(2)}_t(x).
\end{align*}
By the fact that 
$$
\left|\frac{e^{-\int_0^t V(X_s)ds} - 1}{t}\right| \leq \frac{1- e^{-\sup_{y \in (x-\delta(x)/2, x+ \delta(x)/2)}V(y) t}}{t} \leq \sup_{y \in (x-\delta(x)/2, x+ \delta(x)/2)} V(y) < \infty
$$
on $\left\{\tau_x > t\right\}$, by the right continuity of the paths and by the boundedness of $F$, we clearly obtain that $g^{(1)}_t(x) \to 0$ as $t \to 0^{+}$. On the other hand, by the fact that $|e^{-y} - e^{-z}| \leq |y-z|$ for $y,z \geq 0$ and by the strong Markov property, we have 
\begin{align*}
|g^{(2)}_t(x)| & \leq \frac{1}{t} \left\|F\right\|_{\infty} \ex^x \left[\int_0^t V(X_s)ds  ; \tau_{(a,b)} > t > \tau_x \right] \\
& = \frac{1}{t} \left\|F\right\|_{\infty} \left(\ex^x \left[\int_0^{\tau_x} V(X_s)ds  ; \tau_{(a,b)} > t > \tau_x \right] \right. + \left.\ex^x \left[\int_{\tau_x}^t V(X_s)ds  ; \tau_{(a,b)} > t > \tau_x \right]\right) \\
& \leq \frac{1}{t} \left\|F\right\|_{\infty} \left(\sup_{y \in (x-\delta(x)/2, x+ \delta(x)/2)}V(y) \ex^x \left[\tau_x; t >\tau_x \right] \right.+ \left.\ex^x \left[\ex^{X_{\tau_x}} \left[\int_0^t V(X_s)ds \right] ; t > \tau_x \right]\right) \\
& \leq \frac{1}{t} \left\|F\right\|_{\infty} \pr^x (t > \tau_x ) \left(t \sup_{y \in (x-\delta(x)/2, x+ \delta(x)/2)}V(y) + \sup_{y \in \R} \ex^y \left[\int_0^t V(X_s)ds \right]\right)
\end{align*}
By using the fact that $\pr^x (t > \tau_x ) \leq C_{\alpha} t {\delta(x)}^{-\alpha}$ (see \cite[Lemma 4.3]{bib:BSS}) and $V \in \cK^{\alpha}$, it follows that $g^{(2)}_t(x) \to 0$ as $t \to 0^{+}$ . 

We now show that also $\left\|g_t \varphi_1 \right\|_{L^1((a,b))} \to 0$ as $t \to 0^{+}$. Since $V \in \cK^{\alpha}$, there exists a constant $C_{\alpha,V}$ such that for all $t \in (0,1]$ $\sup_{y \in \R} \ex^y\left[\int_0^t V(X_s)ds\right] \leq C_{\alpha,V}$. The above estimates for $g_t^{(1)}(x)$, $g_t^{(2)}(x)$ imply
$$
|g_t^{(1)}(x)| \leq \left\|F\right\|_{\infty} \sup_{y \in (x-\delta(x)/2, x+ \delta(x)/2)} V(y),
$$
$$
|g_t^{(2)}(x)| \leq \left\|F\right\|_{\infty} \left( \sup_{y \in (x-\delta(x)/2, x+ \delta(x)/2)} V(y) + C_{\alpha,V} \delta(x)^{-\alpha} \right),
$$
for $x \in (a,b)$, $t \in (0,1]$. A consequence of the above bounds and the fact that $\varphi_1(x) \leq C_{V,a,b} {\delta(x)}^{\alpha/2}$ for $x \in (a,b)$ (see \cite[Proposition 13 and Remark 11]{bib:Ku2}) is the following estimate for $t \in (0,1]$
\begin{align*}
|g_t(x) \varphi_1(x)| & = (|g_t^{(1)}(x)| + |g_t^{(2)}(x)|) \varphi_1 (x) \\
& \leq C_{V,a,b} \left\|F\right\|_{\infty} \left(2 \sup_{y \in (x-\delta(x)/2, x+ \delta(x)/2)}V(y) + C_{\alpha,V} \delta(x)^{-\alpha} \right) \delta(x)^{\alpha/2} \\ & \leq C_{V,a,b} \left\|F\right\|_{\infty}  \left(V\left(\frac{a+x}{2}\right) \vee V\left(\frac{x+b}{2}\right) +  \delta(x)^{-\alpha/2} \right).
\end{align*}
Recall that the constants may change their value from one use to another. By the fact that $V \in L^1((a,b))$ the right hand side of the above inequality is integrable in the interval $(a,b)$ and the assertion of the lemma follows now from the dominated convergence theorem. 
\end{proof}

\begin{lemma}
\label{lm:rD}
Assume that $\alpha \in (0,2)$. Let $-\infty<a<b<\infty$ and $V \in \cV^{\alpha}((a,b))$ be a nonnegative potential. Then for all $f \in C_{\bound}((a,b))$ we have 
\begin{align}
\label{eq:formD}
\widetilde \cE(f,f) = \frac{\cA_{-\alpha}}{2} \int_{a}^b \int_{a}^b \frac{(f(x)-f(y))^2}{|x-y|^{1+\alpha}} \varphi_1(x) \varphi_1(y) dx dy.
\end{align}
\end{lemma}

\begin{proof}
By \eqref{eq:var1} we have 
\begin{align*}
\widetilde \cE(f,f) & = \lim_{t \to 0^{+}} \frac{1}{t} \int_a^b \left(f(x) - \frac{e^{\lambda_1 t}}{\varphi_1(x)} T_t (\varphi_1 f)(x)\right) f(x) \varphi_1^2(x) dx \\
& = \lim_{t \to 0^{+}} \frac{1}{t} \int_a^b \left(f(x)\varphi_1(x) - e^{\lambda_1 t} T_t (\varphi_1 f)(x)\right) f(x) \varphi_1(x) dx \\
& = \lim_{t \to 0^{+}} \frac{1}{t} \int_a^b \left(f(x)e^{\lambda_1 t} T_t \varphi_1(x) - e^{\lambda_1 t} T_t (\varphi_1 f)(x)\right) f(x) \varphi_1(x) dx \\
& = \lim_{t \to 0^{+}} e^{\lambda_1 t}  \frac{1}{t}\int_a^b  T_t F(x,\:\cdot\:) (x) \varphi_1(x) dx,
\end{align*}
where
$$
F(x,y):= (f^2(x) - f(x) f(y)) \varphi_1(y).
$$
We have 
\begin{align*}
T_t F(x,\:\cdot\:)(x) & = \ex^x \left[e^{-\int_0^t V(X_s)ds} F(x,X_t); \tau_{(a,b)} > t\right] \\ & = P^{(a,b)}_t F(x,\:\cdot\:)(x)
+ \ex^x \left[\left(e^{-\int_0^t V(X_s)ds} - 1\right) F(x,X_t); \tau_{(a,b)} > t\right].
\end{align*}
Thus
\begin{align*}
\widetilde \cE(f,f) & = \lim_{t \to 0^{+}} e^{\lambda_1 t}  \left( \frac{1}{t}\int_a^b P^{(a,b)}_t F(x,\:\cdot\:)(x) \varphi_1(x) dx \right.\\ & \ \ \ \ + \left.  \int_a^b \ex^x \left[\frac{e^{-\int_0^t V(X_s)ds} - 1}{t} F(x,X_t); \tau_{(a,b)} > t\right]\varphi_1(x) dx \right).
\end{align*}
Clearly, by Lemma \ref{lm:FKlevy}, the second integral on the right hand side of the above equality tends to $0$ as $t \to 0^{+}$. It is enough to show that 
\begin{equation}
\begin{split}
\label{eq:formD2}
\lim_{t \to 0^{+}} e^{\lambda_1 t}  \frac{1}{t} & \int_a^b P^{(a,b)}_t F(x,\:\cdot\:)(x) \varphi_1(x) dx 
 = \frac{\cA_{-\alpha}}{2} \int_{a}^b \int_{a}^b \frac{(f(x)-f(y))^2}{|x-y|^{1+\alpha}} \varphi_1(x) \varphi_1(y) dx dy.
\end{split}
\end{equation}
The left hand side of \eqref{eq:formD2} can be rewritten as 
$$
\lim_{t \to 0^{+}} e^{\lambda_1 t} \int_a^b \int_a^b \frac{p_{(a,b)}(t,x,y)}{t} (f^2(x) - f(x)f(y))\varphi_1(y) \varphi_1(x) dy dx,
$$
which by symmetry is equal to
\begin{align}
\label{eq:final}
\lim_{t \to 0^{+}} \frac{e^{\lambda_1 t}}{2} \int_a^b \int_a^b \frac{p_{(a,b)}(t,x,y)}{t} (f(x) - f(y))^2\varphi_1(y) \varphi_1(x) dy dx.
\end{align}
In the light of \eqref{eq:levy2} in order to prove \eqref{eq:formD2} we need only to justify the interchange of the limit and the integral in \eqref{eq:final}. Denote 
$$
\cQ(f,f) = \int_{a}^b \int_{a}^b \frac{(f(x)-f(y))^2}{|x-y|^{1+\alpha}} \varphi_1(x) \varphi_1(y) dx dy.
$$
If $\cQ(f,f) = \infty$, then \eqref{eq:formD2} follows from \eqref{eq:final} by the Fatou lemma. Now let us consider the case $\cQ(f,f) < \infty$. By \eqref{eq:kernest1} and \eqref{eq:kernest} we have for any $t>0$ 
$$
\frac{p_{(a,b)}(t,x,y)}{t} (f(x) - f(y))^2\varphi_1(y) \varphi_1(x) \leq C \frac{(f(x)-f(y))^2}{|x-y|^{1+\alpha}} \varphi_1(x) \varphi_1(y).
$$
Since the integral over $(a,b) \times (a,b)$ of the right hand side of the above inequality is equal to $C \cQ(f,f) < \infty$, the assertion of the lemma follows from \eqref{eq:final} by the dominated convergence theorem.
\end{proof}

\section{Properties of eigenfunctions}

For our convenience we work with the symmetric interval $(-a,a)$, $0<a<\infty$, in the proof of Theorem \ref{th:gsp}. First we need the following auxiliary lemmas. Let us note that our arguments in the proofs of Lemma \ref{lm:midphi}, Lemma \ref{lm:aproks} and Theorem \ref{th:gsp} extend some earlier ideas from \cite{bib:BK3, bib:BKM} and \cite{bib:BM}.

\begin{lemma}
\label{lm:midphi}
Assume that $\alpha \in (0,2)$. Let $0<a<\infty$ and let $V$ be a bounded, symmetric and continuous function on $(-a,a)$ such that $V^{\prime}(x)$ exists in $(-a,a)$ except for at most finite points and $V$ is nonincreasing in $(-a,0)$. Let $n$ be a natural number and let $s_i, t_i$, $i=1,...,n$, be arbitrary positive parameters. For $x \in [-a,a]$ we define 
\begin{align*}
\Phi_n(x; s_1, ..., s_n, t_1, &...,t_n) =  \int_{-a}^a ... \int_{-a}^a \exp\left(-\sum_{i=1}^n t_i V(x_i)\right) \prod_{i=1}^n q(s_i,x_i-x_{i-1})dx_1 ...dx_n,
\end{align*}
where $x_0=x$ and $q$ is given by \eqref{eq:BM}. Then for every $n \geq 1$ and parameters $s_i, t_i$, $i=1,...,n$, $\Phi_n(x; s_1, ..., s_n, t_1, ...,t_n)$ is nondecreasing on $(-a,0)$ and nonincreasing on $(0,a)$ as a function of $x$.
\end{lemma}

\begin{proof}
First note that for every $n \geq 1$ and positive parameters $s_i, t_i$, $i=1,...,n$, $\Phi_n(x; s_1, ..., s_n, t_1, ...,t_n)$ is symmetric and positive on $[-a,a]$. Without loosing generality we assume that $a=1$. We use the induction. Integrating by parts and using the fact that 
$V$ is symmetric, we have
\begin{equation}
\label{eq:firstder}
\begin{split}
  \sqrt{4 \pi s_1} & \frac{d}{d x} \Phi_1(x; s_1,t_1) 
 =  - \int_{-1}^1 e^{-t_1 V(y)} \frac{\partial}{\partial y} e^{-\frac{(y-x)^2}{4s_1}}  dy \\
 = & \ e^{-\frac{(1+x)^2}{4s_1}} \lim_{y \to -1} e^{-t_1 V(y)} - e^{-\frac{(1-x)^2}{4s_1}} \lim_{y \to 1} e^{-t_1 V(y)} 
  + \int_{-1}^1e^{-\frac{(y-x)^2}{4s_1}} \frac{d}{d y} e^{-t_1 V(y)}  dy  \\
 = &  \lim_{y \to 1} e^{-t_1 V(y)} \left(e^{-\frac{(1+x)^2}{4s_1}} - e^{-\frac{(1-x)^2}{4s_1}} \right)
  + \int_0^1  \left( e^{-\frac{(y-x)^2}{4s_1}}
  -  e^{-\frac{(y+x)^2}{4s_1}} \right) \frac{d}{d y} e^{-t_1 V(y)}dy .
\end{split}
\end{equation}
Since
\begin{align}
\label{eq:gauss}
e^{-\frac{(y-x)^2}{4s}} & \geq e^{-\frac{(y+x)^2}{4s}}
\, , & x,y,s > 0
\, ,
\end{align}
and
\begin{align*}
\frac{d}{d y} e^{-t V(y)} = - t e^{-t V(y)} V^{\prime}(y) & \leq 0
\, , &  \ t > 0
\, ,
\end{align*}
for almost all $y \in (0,1)$, the expression on the right hand side of \eqref{eq:firstder} is nonpositive for all $x \in (0,1)$. Thus, for every $s_1, t_1 >0$ the function $\Phi_1(x; s_1,t_1)$ is nonincreasing on $(0,1)$. 

Now suppose that for any positive parameters $s_1,...s_{n-1}, t_1,...,t_{n-1}$ the function $\Phi_{n-1}$ is nonincreasing on $(0,1)$, that is, $\frac{d}{d x} \Phi_{n-1}(x; s_1, ..., s_{n-1}, t_1, ...,t_{n-1}) \leq 0$ for $x \in (0,1)$. Our claim is $\frac{d}{d x} \Phi_n(x; s_1, ..., s_n, t_1, ...,t_n) \leq 0$ for $x \in (0,1)$ and for every positive parameters $s_1,...s_n, t_1,...,t_n$. We observe that
\begin{align*}
\Phi_n(x; s_1, ..., s_n, t_1, ...,t_n) = \frac{1}{\sqrt{4 \pi s_1}} \int_{-1}^1 e^{-\frac{(y-x)^2}{4s_1}} e^{-t_1 V(y)} \Phi_{n-1}(y; s_2, ..., s_n, t_2, ...,t_n) dy
\, .
\end{align*}
From now on, for more simplicity, we omit the parameters $s_i,t_i$ in our notation. Integrating by parts, similarly as before, we have  
\begin{equation}
\label{eq:nder}
\begin{split}
\sqrt{4 \pi s_1} & \frac{d}{d x} \Phi_n(x) 
  = - \int_{-1}^1 \frac{\partial}{\partial y} \left(e^{-\frac{(y-x)^2}{4s_1}}\right) e^{-t_1 V(y)}\Phi_{n-1}(y) dy \\
&  = \Phi_{n-1}(1) \lim_{y \to 1} e^{-t_1 V(y)}\left(e^{-\frac{(1+x)^2}{4s_1}} - e^{-\frac{(1-x)^2}{4s_1}} \right) + \int_{-1}^1 e^{-\frac{(y-x)^2}{4s_1}} \frac{d}{d y} \left(e^{-t_1 V(y)}\Phi_{n-1}(y) \right)dy.
\end{split}
\end{equation}
Observe that by \eqref{eq:gauss} the first term on the right hand side of \eqref{eq:nder} is nonpositive. Thus it is enough to show that for all $x \in (0,1)$ the last integral also is nonpositive. Denote $h(y) = e^{-t_1 V(y)}\Phi_{n-1}(y)$. Recall that $V^{\prime}(y) \geq 0$ for almost all $y \in (0,1)$. By induction hypothesis and by symmetry, we have respectively
$$
h^{\prime}(y) = \frac{d}{d y} \left(e^{-t_1 V(y)}\Phi_{n-1}(y) \right) = -t_1 V^{\prime}(y) e^{-t_1 V(y)}\Phi_{n-1}(y) + e^{-t_1 V(y)} \frac{d}{d y} \Phi_{n-1}(y) \leq 0
$$
and
\begin{align*}
h^{\prime}(y) = - h^{\prime}(-y)
\, ,
\end{align*}
for almost all $y \in (0,1)$. This and \eqref{eq:gauss} give for all $x \in (0,1)$, almost all $y \in (0,1)$ and $t_1, s_1 >0$
\begin{align*}
 e^{-\frac{(y-x)^2}{4s_1}} h^{\prime}(y)
+ e^{-\frac{(y+x)^2}{4s_1}} h^{\prime}(-y) 
 = e^{-\frac{(y-x)^2}{4s_1}} h^{\prime}(y)
- e^{-\frac{(y+x)^2}{4s_1}} h^{\prime}(y) 
 = \left( e^{-\frac{(y-x)^2}{4s_1}} - e^{-\frac{(y+x)^2}{4s_1}}\right)h^{\prime}(y) \leq 0
\, .
\end{align*}
Integrating this inequality on $[0,1]$, we get 
\begin{align*}
 \int_{-1}^1 e^{-\frac{(y-x)^2}{4s_1}} \frac{d}{d y} \left(e^{-t_1 V(y)}\Phi_{n-1}(y)\right) dy \leq 0
\, ,
\end{align*}
for $x \in (0,1)$, which ends the proof.
\end{proof}

\begin{lemma}
\label{lm:aproks}
Assume that $\alpha \in (0,2)$. Let $0<a<\infty$ and $V \in \cV^{\alpha}((-a,a))$. For $t>0$ denote
$$
g_t(x) = \ex^x\left[\exp\left(-\int_0^t V(X_s)ds\right);  t< \tau_{(-a,a)}\right], \ \ \ \ \ \ \ \ \ x \in (-a,a).
$$
Then for every fixed $t>0$ the functions $g_t(x)$ is nondecreasing in $(-a,0)$ and nonincreasing in $(0,a)$. 
\end{lemma}

\begin{proof}
Let $V_k(x) = V(x) \wedge k$, $x \in (-a,a)$, $k \geq 1$. First note that for every $x \in (-a,a)$ and $t>0$ 
$$
\int_0^t V_k(X_s)ds \rightarrow \int_0^t V(X_s)ds
$$
as $k \to \infty$, $\pr^x$-almost surely, by the monotone convergence theorem. Thus for all $x \in (-a,a)$ and $t > 0$ we have 
{\small$$
\ex^x\left[\exp\left(-\int_0^t V(X_s)ds\right);  t< \tau_{(-a,a)}\right] = \lim_{k \to \infty} \ex^x\left[\exp\left(-\int_0^t V_k(X_s)ds\right);  t< \tau_{(-a,a)}\right]
$$}
as a consequence of the fact that
$$
\exp\left(-\int_0^t V_k(X_s)ds\right) \leq \exp(-(\inf V \wedge 1) t), \ \ \ \ \ \ t >0, \ \ \ k \geq 1,
$$
and the bounded convergence theorem. Thus, it suffices to show that for any $k \geq 1$
$$
\ex^x\left[\exp\left(-\int_0^t V_k(X_s)ds\right);  t< \tau_{(-a,a)}\right]
$$ 
is nondecreasing in $(-a,0)$ and nonincreasing in $(0,a)$ as a function of variable $x$, for each fixed $t >0$. 

Fix now $k \geq 1$. By using the fact that $\pr^x(X_{\tau_{(-a,a)}} \in \left\{-a, a\right\}) = 0$ for $x \in (-a,a)$ and the paths are c\`adl\`ag, the boundedness and continuity of $V_k$, and the Markov property of the process $(X_t)_{t \geq 0}$, we have
\begin{align*}
\ex^x & \left[\exp\left(-\int_0^t V_k(X_s)ds\right); t< \tau_{(-a,a)}\right] \\ & = \ex^x\left[\exp\left(-\int_0^t V_k(X_s)ds\right); X_s \in (-a,a), \forall 0\leq s\leq t\right] \\
                          & = \lim_{n \rightarrow \infty}  \ex^x\left[\exp\left(-\frac{t}{2^n} \sum_{i=1}^{2^n} V_k(X_{it/2^n})\right); X_{it/2^n} \in (-a,a), i=1,2,...,2^n\right]\\
                          & = \lim_{n \rightarrow \infty} \int_{-a}^a ... \int_{-a}^a \prod_{i=1}^{2^n} \exp\left(- \frac{t}{2^n} V_k(x_i)\right) p(t/2^n, x_i-x_{i-1})dx_1 ...dx_{2^n},
\end{align*}
where $x_0 = x$. Moreover, by the subordination formula \eqref{eq:subfor} and Fubini's theorem, the last multiple integral can be rewritten as 
\begin{equation}
\begin{split}
\label{eq:lett}
                          & \int_{-\infty}^{\infty} ... \int_{-\infty}^{\infty} \left(\int_{-a}^a ... \int_{-a}^a \prod_{i=1}^{2^n} \exp\left(- \frac{t}{2^n} V_k(x_i)\right) q(s_i, x_i-x_{i-1})dx_1 ...dx_{2^n} \right) \\
                          & \ \times \prod_{i=1}^{2^n} g_{\alpha/2}(t/2^n, s_i)ds_1 ... ds_{2^n} \\
                          & = \int_{-\infty}^{\infty} ... \int_{-\infty}^{\infty} \Phi_{2^n}(x;s_1,...,s_{2^n},t/2^n,...,t/2^n) \prod_{i=1}^{2^n} g_{\alpha/2}(t/2^n, s_i)ds_1 ... ds_{2^n}
\, ,
\end{split}
\end{equation}
where
$$
\Phi_{2^n}(x;s_1,...,s_{2^n},t_1,...,t_{2^n}):= \int_{-a}^a ... \int_{-a}^a \prod_{i=1}^{2^n} \exp\left(- t_i V_k(x_i)\right) q(s_i, x_i-x_{i-1})dx_1 ...dx_{2^n}, 
$$ 
with $x_0=x$. By Lemma \ref{lm:midphi} we obtain that for every natural $n$ and for any positive parameters $t_i, s_i$, $i=1,...,2^n$, the function $\Phi_{2^n}$ is nondecreasing in $(-a,0)$ and nonincreasing in $(0,a)$. Letting now $n \to \infty$ in \eqref{eq:lett}, we conclude that the same is true for $\ex^x\left[\exp\left(-\int_0^t V_k(X_s)ds\right);  t< \tau_{(-a,a)}\right]$, for any $k \geq 1$ and for each fixed $t>0$. Thus the proof is complete.  
\end{proof}

\begin{proof}[Proof of Theorem \ref{th:gsp}]
Let $0<a<\infty$ and $V \in \cV^{\alpha}((-a,a))$ be fixed. First we show a symmetry. Suppose contrary that $\varphi_1$ is not symmetric. Thus $\widehat \varphi_1(x) := \varphi_1(x) + \varphi_1(-x)$ is also an eigenfunction of $-L$ corresponding to the eigenvalue $\lambda_1$ such that $\widehat \varphi_1 \notin \spann(\varphi_1)$. This gives a contradiction, because $\lambda_1$ has muliplicity one. Thus $\varphi_1$ is symmetric in $(-a,a)$.

Let now $P_{\varphi_1}:L^2((-a,a)) \rightarrow L^2((-a,a))$ be the projection onto $\spann(\varphi_1)$. By the fact that $e^{\lambda_1 t} T_t P_{\varphi_1} = P_{\varphi_1} e^{\lambda_1 t} T_t = P_{\varphi_1}$, we have for $t > 2$
\begin{align*}
\sup_{x \in (-a,a)}\left|(e^{\lambda_1 t} \ T_t - P_{\varphi_1})\1_{(-a,a)}(x) \right| & \leq 2 a \left\|e^{\lambda_1 t} \ T_t - P_{\varphi_1} \right\|_{1,\infty} \\ & \leq 2 a e^{2 \lambda_1} \left\|T_1\right\|_{2, \infty }\left\|e^{\lambda_1 (t-2)} \ T_{t-2} - P_{\varphi_1} \right\|_{2}\left\|T_1\right\|_{1, 2}.
\end{align*}
Since by the spectral theorem 
$$
\left\|e^{\lambda_1 t} \ T_t - P_{\varphi_1} \right\|_{2} \leq e^{-\left(\lambda_2 - \lambda_1\right)t},
$$
we obtain 
{\small
$$
\varphi_1(x) = \lim_{t \to \infty} C_V \ e^{\lambda_1 t} \ T_t \1_{(-a,a)}(x) = \lim_{t \to \infty} C_V \ e^{\lambda_1 t} \ \ex^x\left[\exp\left(-\int_0^t V(X_s)ds\right);  t< \tau_{(-a,a)}\right], 
$$
}
uniformly in $x \in (-a,a)$, with $C_V = \left(\left\|\varphi_1\right\|_{1}\right)^{-1}$. Now the assertion of the theorem follows from Lemma \ref{lm:aproks}.
\end{proof}
  
We need the following auxiliary lemma. It is a version of \cite[Lemma 5.2]{bib:BNK} and \cite[Lemma 10]{bib:BJ}.

\begin{lemma}
\label{lm:Greender}
Let $\alpha \in (1,2)$. Let $-\infty<a<b<\infty$ and let $f \in L^1((a,b))$ be a function such that for every interval $[c,d] \subset (a,b)$ we have $\sup_{x \in [c,d]} |f(x)| < \infty$.
Then
\begin{align}
\label{eq:Gr1}
\frac{d}{dx} G_{(a,b)} f(x) = \int_a^b \frac{\partial}{\partial x} G_{(a,b)}(x,y) f(y) dy, \ \ \ \ \ \ \ x \in (a,b),
\end{align}
and for every interval $[c,d] \subset (a,b)$ there is a constant $C_{\alpha,f,a,b,c,d} < \infty$ such that 
\begin{align}
\label{eq:Gr2}
\left|\frac{d}{dx} G_{(a,b)} f(x)\right| \leq C_{\alpha,f,a,b,c,d}, \ \ \ \ \ \ \ \ \ x \in [c,d].
\end{align}
\end{lemma}

\begin{proof}
Recall that $\alpha \in (1,2)$. By \cite[Formula (5)]{bib:BNK} we have
$$
G_{(a,b)}(x,y) = K^{(\alpha)}(x-y) - H(x,y), \ \ \ \ \ \ \ \ x, y \in (a,b), \ x \neq y,
$$
where $K^{(\alpha)}(x-y) = (2\Gamma(\alpha) \cos(\pi\alpha/2))^{-1} |x-y|^{\alpha-1}$ and $H(x,y) = \ex^x \left[K^{(\alpha)}(X_{\tau_{(a,b)}} - y)\right]$. In view of this equality we have 
\begin{align*}
\frac{d}{dx} G_{(a,b)} f(x) = & \lim_{h \to 0} \int_a^b \frac{K^{(\alpha)}(x+h-y) - K^{(\alpha)}(x-y)}{h}f(y)dy \\
                            & + \lim_{h \to 0} \int_a^b \frac{H(x+h-y) - H(x-y)}{h}f(y)dy, \ \ \ \ \ \ x \in (a,b).
\end{align*}
First notice that both partial derivatives $\frac{\partial}{\partial x} K^{(\alpha)}(x-y)$, $x \neq y$, and $\frac{\partial}{\partial x} H(x,y)$ exist (see (10) in \cite{bib:BNK}). For $x \in (a,b)$ denote $\delta(x) = (b-x) \wedge (x-a)$. From \cite[Lemma 3.2]{bib:BNK} we have $\left|\frac{\partial}{\partial x} H(x,y)\right| \leq C_{\alpha,a,b} \delta(x)^{-1}$. It follows that 
\begin{equation}
\begin{split}
\label{eq:derbound}
\left|\frac{\partial}{\partial x} G_{(a,b)}(x,y)\right| & \leq \left|\frac{\partial}{\partial x} K^{(\alpha)}(x-y)\right| + \left|\frac{\partial}{\partial x} H(x,y)\right| \\
& \leq C_{\alpha}  |x-y|^{\alpha-2} + C_{\alpha,a,b} \delta(x)^{-1}, \ \ \ \ \ \ \  x, y \in (a,b), \ x \neq y.
\end{split}
\end{equation}
This inequality and properties of $f$ imply that for each fixed $x \in (a,b)$ the integral on the right hand side of \eqref{eq:Gr1} is absolutely convergent. Thus, to obtain \eqref{eq:Gr1} it is enough to show that 
\begin{align}
\label{eq:eqq1}
\lim_{h \to 0} \int_a^b |F^{(1)}_h(x,y)||f(y)|dy + \lim_{h \to 0} \int_a^b  |F^{(2)}_h(x,y)||f(y)|dy = 0, \ \ \ \ \ \ x \in (a,b),
\end{align}
where
$$
F^{(1)}_h(x,y) = \frac{K^{(\alpha)}(x+h-y) - K^{(\alpha)}(x-y)}{h} - \frac{\partial}{\partial x} K^{(\alpha)}(x-y),
$$
$$
F^{(2)}_h(x,y) = \frac{H(x+h-y) - H(x-y)}{h} - \frac{\partial}{\partial x} H(x,y).
$$
Fix now $x \in (a,b)$. Let $|h|< \delta(x)/4$. From \cite[Lemma 3.2]{bib:BNK} and Lagrange's theorem we obtain
\begin{align}
\label{eq:eqq3}
|F^{(2)}_h(x,y)| \leq C_{\alpha,a,b} \delta(x)^{-1}.
\end{align}
This estimate and the fact that $f \in L^1((a,b))$ give that 
$$
\lim_{h \to 0} \int_a^b  |F^{(2)}_h(x,y)||f(y)|dy = 0
$$ 
by the dominated convergence theorem. 

It suffices to show that
\begin{align}
\label{eq:twostars}
\lim_{h \to 0} \int_a^b  |F^{(1)}_h(x,y)||f(y)|dy = 0
\end{align}
for each fixed $x \in (a,b)$. Let $\beta \in (0,1/2)$ and 
\begin{equation}
\begin{split}
\label{eq:eqq4}
\int_a^b |F^{(1)}_h(x,y)||f(y)|dy = & \int_{(x-\beta \delta(x),x+\beta \delta(x))} |F^{(1)}_h(x,y)||f(y)| dy \\ & + \int_{(a,b)\cap(x-\beta \delta(x),x+\beta \delta(x))^c}|F^{(1)}_h(x,y)||f(y)| dy .
\end{split}
\end{equation}
Fix $x \in (a,b)$ and $\varepsilon > 0$. We will show that for sufficiently small $|h|$ the left hand side of \eqref{eq:eqq4} is smaller than $\varepsilon$. Let $[c,d] \subset (a,b)$ be such that $x \in (c,d)$. Denote $M = \sup_{x \in [c,d]} |f(x)|$. Let $\beta$ be small enough so that $(x-\beta \delta(x), x+\beta \delta(x)) \subset [c,d]$. It is known (see \cite[proof of Lemma 5.2]{bib:BNK}) that 
\begin{align}
\label{eq:eqq2}
|F^{(1)}_h(x,y)| \leq C_{\alpha} \left(|x+h-y|^{\alpha-2} \vee |x-y|^{\alpha-2}\right), \ \ \ \ \ \ y \in (a,b), \ y \neq x, \ y \neq x+h.
\end{align}
Hence for any $h \in \R$
\begin{align*}
\int_{x-\beta \delta(x)}^{x+\beta \delta(x)} |F^{(1)}_h(x,y)||f(y)|dy & \leq M C_{\alpha} \int_{x-\beta \delta(x)}^{x+\beta \delta(x)} \left(|x+h-y|^{\alpha-2} + |x-y|^{\alpha-2}\right) dy \\ & \leq 2 M C_{\alpha} \int_{-\beta \delta(x)}^{\beta \delta(x)} |y|^{\alpha - 2} dy.
\end{align*}
It is clear that there exists $\beta$ small enough so that the above integral is smaller than $\varepsilon/2$. Let us fix such $\beta$. Clearly, $F^{(1)}_h(x,y) \to 0$ as $h \to 0$ for any $x \neq y$. By \eqref{eq:eqq2}
$$
|F^{(1)}_h(x,y)| \leq C_{\alpha} (2 \beta^{-1} \delta(x)^{-1} \vee 1) ,
$$
for $y \in (a,b)\cap(x-\beta \delta(x),x+ \beta \delta(x))^c$ and $h \in (- \beta \delta(x)/2, \beta \delta(x)/2)$. Since $f \in L^1((a,b))$, the second integral on the right hand side of  \eqref{eq:eqq4} tends to $0$ as $h$ tends to $0$ by the bounded convergence theorem. Hence for $|h|$ sufficiently small the second integral on the right hand side of \eqref{eq:eqq4} is smaller than $\varepsilon/2$. This finishes the proof \eqref{eq:twostars}. Thus \eqref{eq:Gr1} is proved. The boundedness property \eqref{eq:Gr2} is a simple consequence of the estimate \eqref{eq:derbound} and the properties of $f$. 
\end{proof}

\begin{proof}[Proof of Theorem \ref{th:gsp1}]
Let $-\infty<a<b<\infty$ and $V \in \cV^{\alpha}((a,b))$. The starting point of the proof are the eigenequations
\begin{align}
\label{eq:eigeneq}
T_t \;\varphi_n = e^{-\lambda_n t} \; \varphi_n, \ \ \ \ \ \ n \geq 1.
\end{align}
Since we do not exclude the case that $V$ is signed potential, it may happen that $\lambda_n < 0$ for finitely many $n$. Put
$$
\eta= 
\left\{
\begin{array}{ll} 0, & \text{if} \ \ \ \ \inf V >0,\\
1,                 & \text{if} \ \ \ \  \inf V = 0, \\
-2 \inf V,    & \text{if} \ \ \ \ \inf V < 0.
\end{array} \right. 
$$
Denote $V_{\eta} =  V + \eta$. Then $V_{\eta} > 0$ and $((a,b),V_{\eta})$ is gaugeable (see p. 7). By \eqref{eq:eigeneq} we clearly have
$$
e^{-(\lambda_n + \eta) t} \; \varphi_n(x) = e^{-\eta t} T_t \;\varphi_n(x) = \ex^x \left[e^{-\int_0^t V_{\eta}(X_s)ds}\varphi_n(X_t); \tau_{(a,b)} > t \right], \ \ \ \ \ x \in (a,b).
$$
Using the fact that $\lambda_n + \eta >0$, $n \geq 1$, and integrating over $t$ the above equations we obtain
$$
\varphi_n(x) = (\lambda_n + \eta) G^{V_{\eta}}_{(a,b)} \varphi_n(x), \ \ \ x \in (a,b), \ n\geq 1.
$$ 
Applying now the perturbation formula \eqref{eq:pertform} to this equality we get
$$
\varphi_n(x) = (\lambda_n + \eta) G_{(a,b)} \varphi_n(x) + G_{(a,b)} (V_{\eta} \varphi_n)(x), \ \ \ x \in (a,b), \ n\geq 1,
$$
which can be rewritten as
$$
\varphi_n(x) = (\lambda_n + \eta)  \int_{a}^b G_{(a,b)}(x,y) \varphi_n(y) dy  + \int_{a}^b G_{(a,b)}(x,y) V_{\eta}(y) \varphi_n(y)dy.
$$
Since $\left\|\varphi_n\right\|_{\infty}<\infty$, $V_{\eta} \in L^1((a,b))$ and is continuous in $(a,b)$, the assumptions of Lemma \ref{lm:Greender} are satisfied. Thus for $x \in (a,b)$ we have
$$
\frac{d}{dx} \varphi_n(x) = (\lambda_n + \eta)  \int_{a}^b \frac{\partial}{\partial x}G_{(a,b)}(x,y) \varphi_n(y) dy  + \int_{a}^b \frac{\partial}{\partial x} G_{(a,b)}(x,y) V_{\eta}(y) \varphi_n(y)dy.
$$
A direct consequence of Lemma \ref{lm:Greender} is that also for any interval $[c,d] \subset (a,b)$ there is a constant $C_{V,\alpha,n,a,b,c,d}$ such that for all $x \in [c,d]$ we have
$$
\left|\frac{d}{dx} \varphi_n(x)\right| \leq C_{V,\alpha,n,a,b,c,d}.
$$
\end{proof}

\section{Lower bound for $\lambda_{*} - \lambda_1$}

\begin{proof}[Proof of Theorem \ref{th:sga}]
Let $0<a<\infty$. With no loss of generality we provide the arguments for the symmetric interval $(-a,a)$ only. Let $V \in \cV^{\alpha}((-a,a))$. Recall that our orthonormal basis $\left\{\varphi_n\right\}$ is chosen so that $\varphi_n$ are either symmetric or antisymmetric. Let $n_0$ be the smallest natural number such that $\varphi_{n_0}$ is antisymmetric in $(-a,a)$. Thus $\varphi_{*} = \varphi_{n_0}$. Let $f = \varphi_{*} / \varphi_1 = \varphi_{n_0} / \varphi_1$. For every $\varepsilon \in (0,a)$ we have
\begin{align*}
\int_{-a}^a \int_{-a}^a \frac{(f(x)-f(y))^2}{|x-y|^{1+\alpha}} \varphi_1(x) \varphi_1(y) dx dy & \geq \int_{\varepsilon}^a \int_{\varepsilon}^a \frac{(f(x)-f(y))^2}{|x-y|^{1+\alpha}} \varphi_1(x) \varphi_1(y) dx dy \\
& + \int_{\varepsilon}^a \int_{-a}^{-\varepsilon} \frac{(f(x)-f(y))^2}{|x-y|^{1+\alpha}} \varphi_1(x) \varphi_1(y) dx dy \\
& + \int_{-a}^{-\varepsilon} \int_{-a}^{-\varepsilon} \frac{(f(x)-f(y))^2}{|x-y|^{1+\alpha}} \varphi_1(x) \varphi_1(y) dx dy \\
& + \int_{-a}^{-\varepsilon} \int_{\varepsilon}^a \frac{(f(x)-f(y))^2}{|x-y|^{1+\alpha}} \varphi_1(x) \varphi_1(y) dx dy.
\end{align*}
Simple changes of variables in the last three integrals and the fact that $f$ is antisymmetric give that the last sum can be transformed to 
\begin{align*}
2 \int_{\varepsilon}^a \int_{\varepsilon}^a \left(\frac{(f(x)-f(y))^2}{|x-y|^{1+\alpha}} + \frac{(f(x)+f(y))^2}{(x+y)^{1+\alpha}}\right) \varphi_1(x) \varphi_1(y) dx dy.
\end{align*}
Clearly, this is bigger or equal to
\begin{align*} 
2 \int_{\varepsilon}^a \int_{\varepsilon}^a \frac{(f(x)-f(y))^2 + (f(x)+f(y))^2}{(x+y)^{1+\alpha}} & \varphi_1(x) \varphi_1(y) dx dy \\
& = 4 \int_{\varepsilon}^a \int_{\varepsilon}^a \frac{f^2(x) + f^2(y)}{(x+y)^{1+\alpha}} \varphi_1(x) \varphi_1(y) dx dy,
\end{align*}
which, by symmetry, is equal to 
$$
8 \int_{\varepsilon}^a \int_{\varepsilon}^a \frac{f^2(x)}{(x+y)^{1+\alpha}} \varphi_1(x) \varphi_1(y) dx dy.
$$
Thus, by Theorem \ref{th:gsp}, we have 
\begin{align*}
\int_{-a}^a \int_{-a}^a \frac{(f(x)-f(y))^2}{|x-y|^{1+\alpha}} \varphi_1(x) \varphi_1(y) dx dy & \geq 8 \int_{\varepsilon}^a \int_{\varepsilon}^a \frac{f^2(x)}{(x+y)^{1+\alpha}} \varphi_1(x) \varphi_1(y) dy dx \\
& \geq 8 \int_{\varepsilon}^a \int_{\varepsilon}^x \frac{f^2(x)}{(x+y)^{1+\alpha}} \varphi_1(x) \varphi_1(y) dy dx \\
& \geq 8 \int_{\varepsilon}^a \int_{\varepsilon}^x dy \frac{f^2(x)}{(2x)^{1+\alpha}} \varphi_1^2(x) dx  \\
& \geq \frac{4}{(2a)^{\alpha}} \int_{\varepsilon}^a  \frac{x-\varepsilon}{x}f^2(x) \varphi_1^2(x) dx.
\end{align*}
Now, letting $\varepsilon \rightarrow 0$, we obtain
\begin{align*}
\int_{-a}^a \int_{-a}^a \frac{\left(f(x)-f(y)\right)^2}{|x-y|^{1+\alpha}} & \varphi_1(x) \varphi_1(y) dx dy \\  \geq & \frac{4}{(2a)^{\alpha}} \int_0^a  f^2(x)\varphi_1^2(x) dx = \frac{2}{(2a)^{\alpha}} \int_{-a}^a  f^2(x)\varphi_1^2(x) dx =\frac{2}{(2a)^{\alpha}}.
\end{align*}
Since $f = \varphi_{*}/ \varphi_1$ is antisymmetric, the assertion of Theorem \ref{th:sga} follows simply from Proposition \ref{prop:var}.

\end{proof}

\section{Crucial inequality}
\label{sec:crucineq}
\begin{proof}[Proof of Theorem \ref{th:fpi}]
If $f(a)=f(b)=0$, then the inequality \eqref{eq:fpi} is trivial. We will only consider the case $f(a)=0$, $f(b) \neq 0$. The case $f(a)\neq0$, $f(b)= 0$, is similar. Without loosing of generality we may and do assume that $f(b) > 0$. For more clarity we divide the proof into the two parts. 
\smallskip

\noindent
\textbf{(Part 1)} In this part we prove the theorem for $a=0$ and $b=1$. Observe that by the scaling property of the claimed inequality \eqref{eq:fpi}, we may assume that $f(1)=1$. Let $c < 1/2$ be a constant. It will be chosen later. Since our argument in this part is reccurent, it is divided into the several steps.
\smallskip

\noindent
\emph{(Step 1)} Denote 
$$
\left\{
\begin{array}{l} 
a_1=0                \\
b_1=1                  
\end{array} \right. , \ \ \ \ \ \ 
\left\{
\begin{array}{l} 
f_{a_1}=0              \\
f_{b_1}=1                  
\end{array} \right. , \ \ \ \ \ \ 
\left\{
\begin{array}{l} 
x_1=a_1 + c              \\
y_1=b_1 -c                  
\end{array} \right. , \ \ \ \ \ \ 
\left\{
\begin{array}{l} 
f_{x_1}=f_{a_1}+\frac{1}{3}              \\
f_{y_1}=f_{b_1} - \frac{1}{3}                 
\end{array} \right. ,
$$
and define 
$$
m_1 = \min\left\{x \in (a_1,b_1): f(x) = f_{x_1} \right\} \ \ \ \text{and} \ \ \ \ M_1 = \max \left\{x \in (a_1, b_1): f(x) = f_{y_1}\right\}.
$$
Clearly, $f_{y_1}-f_{x_1} = 1/3$. If $m_1 \notin (a_1,x_1)$ and $M_1 \notin (y_1, b_1)$, then we have 
\begin{align*}
\int_0^1 \int_0^1 \frac{(f(x)-f(y))^2}{|x-y|^{1+\alpha}}dxdy & \geq \int_{a_1}^{x_1} \int_{y_1}^{b_1} \frac{(f(x)-f(y))^2}{|x-y|^{1+\alpha}}dxdy \\ 
& \geq \frac{(f_{y_1}-f_{x_1})^2 (b_1-y_1)(x_1-a_1)}{|b_1-a_1|^{1+\alpha}} = \left(\frac{c}{3}\right)^2.
\end{align*}

If $m_1 \in (a_1,x_1)$ or $M_1 \in (y_1, b_1)$, we consider the next step. 
\smallskip

\noindent
\emph{(Step 2)} If $m_1 \in (a_1,x_1)$ let us take 
$$
\left\{
\begin{array}{l} 
a_2=a_1                \\
b_2=x_1                  
\end{array} \right. , \ \ \ \ \ \ 
\left\{
\begin{array}{l} 
f_{a_2}=f_{a_1}             \\
f_{b_2}=f_{x_1}                  
\end{array} \right. , \ \ \ \ \ \ 
\left\{
\begin{array}{l} 
x_2=a_2 + c^2              \\
y_2=b_2 - c^2                  
\end{array} \right. , \ \ \ \ \ \ 
\left\{
\begin{array}{l} 
f_{x_2}=f_{a_2}+ \left(\frac{1}{3}\right)^2             \\
f_{y_2}=f_{b_2} - \left(\frac{1}{3}\right)^2                 
\end{array} \right. .
$$
If $m_1 \notin (a_1,x_1)$ and $M_1 \in (y_1, b_1)$ let us take
$$
\left\{
\begin{array}{l} 
a_2=y_1                \\
b_2=b_1                  
\end{array} \right. , \ \ \ \ \ \ 
\left\{
\begin{array}{l} 
f_{a_2}=f_{y_1}             \\
f_{b_2}=f_{b_1}                  
\end{array} \right. , \ \ \ \ \ \ 
\left\{
\begin{array}{l} 
x_2=a_2 + c^2              \\
y_2=b_2 - c^2                  
\end{array} \right. , \ \ \ \ \ \ 
\left\{
\begin{array}{l} 
f_{x_2}=f_{a_2}+ \left(\frac{1}{3}\right)^2             \\
f_{y_2}=f_{b_2} - \left(\frac{1}{3}\right)^2                 
\end{array} \right. .
$$
Define 
$$
m_2 = \min\left\{x \in (a_2,b_2): f(x) = f_{x_2} \right\} \ \ \ \ \text{and} \ \ \ \ M_2 = \max \left\{x \in (a_2, b_2): f(x) = f_{y_2}\right\}.
$$
Clearly, $f_{y_2}-f_{x_2} = 1/3^2$, $b_2-a_2 = c$. When $m_2 \notin (a_2,x_2)$ and $M_2 \notin (y_2, b_2)$, we have 
\begin{align*}
\int_0^1 \int_0^1 \frac{(f(x)-f(y))^2}{|x-y|^{1+\alpha}}dxdy & \geq \int_{a_2}^{x_2} \int_{y_2}^{b_2} \frac{(f(x)-f(y))^2}{|x-y|^{1+\alpha}}dxdy \\ 
& \geq \frac{(f_{y_2}-f_{x_2})^2 (b_2-y_2)(x_2-a_2)}{|b_2-a_2|^{1+\alpha}} = \frac{c^4}{3^4 c^{1+\alpha}} = \left(\frac{c}{3}\right)^2  \frac{1}{3^2 c^{\alpha-1}}.
\end{align*}
If $m_2 \in (a_2,x_2)$ or $M_2 \in (y_2, b_2)$, we consider the next step. 
\smallskip

Suppose that after $n-1$ steps $m_{n-1} \in (a_{n-1},x_{n-1})$ or $M_{n-1} \in (y_{n-1}, b_{n-1})$. Then let us consider the next step.
\smallskip

\noindent
\emph{(Step n)} If $m_{n-1} \in (a_{n-1},x_{n-1})$ let us take 
$$
\left\{
\begin{array}{l} 
a_n=a_{n-1}                \\
b_n=x_{n-1}                  
\end{array} \right. , \ \ \ \ \ \ 
\left\{
\begin{array}{l} 
f_{a_n}=f_{a_{n-1}}             \\
f_{b_n}=f_{x_{n-1}}                  
\end{array} \right. , \ \ \ \ \ \ 
\left\{
\begin{array}{l} 
x_n=a_n + c^n              \\
y_n=b_n - c^n                  
\end{array} \right. , \ \ \ \ \ \ 
\left\{
\begin{array}{l} 
f_{x_n}=f_{a_n}+ \left(\frac{1}{3}\right)^n             \\
f_{y_n}=f_{b_n} - \left(\frac{1}{3}\right)^n                 
\end{array} \right. .
$$ 
If $m_{n-1} \notin (a_{n-1},x_{n-1})$ and $M_{n-1} \in (y_{n-1}, b_{n-1})$ let us take
$$
\left\{
\begin{array}{l} 
a_n=y_{n-1}                \\
b_n=b_{n-1}                  
\end{array} \right. , \ \ \ \ \ \ 
\left\{
\begin{array}{l} 
f_{a_n}=f_{y_{n-1}}             \\
f_{b_n}=f_{b_{n-1}}                  
\end{array} \right. , \ \ \ \ \ \ 
\left\{
\begin{array}{l} 
x_n=a_n + c^n              \\
y_n=b_n - c^n                  
\end{array} \right. , \ \ \ \ \ \ 
\left\{
\begin{array}{l} 
f_{x_n}=f_{a_n}+ \left(\frac{1}{3}\right)^n             \\
f_{y_n}=f_{b_n} - \left(\frac{1}{3}\right)^n                 
\end{array} \right. .
$$ 
Define 
$$
m_n = \min\left\{x \in (a_n,b_n): f(x) = f_{x_n} \right\} \ \ \ \text{and} \ \ \ M_n = \max \left\{x \in (a_n, b_n): f(x) = f_{y_n}\right\}.
$$
Of course, $f_{y_n}-f_{x_n} = 1/3^n$, $b_n-a_n = c^{n-1}$. If $m_n \notin (a_n,x_n)$ and $M_n \notin (y_n, b_n)$, then we have 
\begin{align*}
\int_0^1 \int_0^1 \frac{(f(x)-f(y))^2}{|x-y|^{1+\alpha}}&dxdy  \geq \int_{a_n}^{x_n} \int_{y_n}^{b_n} \frac{(f(x)-f(y))^2}{|x-y|^{1+\alpha}}dxdy \\ 
& \geq \frac{(f_{y_n}-f_{x_n})^2 (b_n-y_n)(x_n-a_n)}{|b_n-a_n|^{1+\alpha}} = \frac{c^{2n}}{3^{2n} c^{(n-1)(1+\alpha)}}  = \left(\frac{c}{3}\right)^2  \left(\frac{1}{3^2 c^{\alpha-1}}\right)^{n-1}.
\end{align*}
If $m_n \in (a_n,x_n)$ or $M_n \in (y_n, b_n)$, then we consider the $n+1$ step. 

Recalling that $\alpha \in (1,2)$ and choosing $c$ to be $9^{\frac{-1}{\alpha-1}}$, we obtain that
$$
\left(\frac{1}{3^2 c^{\alpha-1}}\right)^{n} = 1 \ \ \ \ \ \ \text{for all} \ \ \ n \geq 0.
$$
It is enough to see that there exists $n_0 \geq 1$ such that $m_{n_0} \notin (a_{n_0},x_{n_0})$ and $M_{n_0} \notin (y_{n_0}, b_{n_0})$. Indeed, in this case we have
\begin{align*}
\int_0^1 \int_0^1 \frac{(f(x)-f(y))^2}{|x-y|^{1+\alpha}}dxdy \geq \left(\frac{c}{3}\right)^2  \left(\frac{1}{3^2 c^{\alpha-1}}\right)^{n_0-1} = \left(\frac{c}{3}\right)^2 = \left(\frac{1}{9}\right)^{\frac{\alpha+1}{\alpha-1}} .
\end{align*} 

Suppose contrary that for all $n \geq 1$ we have $m_n \in (a_n,x_n)$ or $M_n \in (y_n, b_n)$. This means that there exists a decreasing sequence of intervals $\left\{(a_n,b_n)\right\}_{n \geq 1}$ such that $|b_n-a_n| = \left(\frac{1}{9}\right)^{\frac{n-1}{\alpha-1}}$. By the fact that $f$ is a Lipschitz function on $[a,b]$, there is a constant $C$ such that 
$$
\left(\frac{1}{3}\right)^n=|f(M_n) - f(m_n)| \leq C |M_n - m_n| \leq C |b_n-a_n| = C \left(\frac{1}{9}\right)^{\frac{n-1}{\alpha-1}}, \ \ \ \ n \geq 1,
$$
which gives a contradiction. Thus the inequality \eqref{eq:fpi} is true for $a=0$ and $b=1$ with constant $C^{(4)}_{\alpha} = \left(\frac{1}{9}\right)^{\frac{\alpha+1}{\alpha-1}}$, and the Part 1 is complete.
\smallskip

\noindent
\textbf{(Part 2)} We now prove the theorem for an arbitrary interval $(a,b)$, $-\infty < a < b < \infty$. First note that if $f$ is a Lipschitz function in the interval $[a,b]$ such that $f(a)=0$, then the function $\tilde f(x) := f((b-a)x + a)$ is a Lipschitz function in $[0,1]$ such that $\tilde f(0) = 0$. Thus, by trivial change of variables, we have
\begin{align*}
\int_a^b \int_a^b \frac{(f(x)-f(y))^2}{|x-y|^{1+\alpha}} dx dy & = \int_a^b \int_a^b \frac{\left(\tilde f\left(\frac{x-a}{b-a}\right)- \tilde f\left(\frac{y-a}{b-a}\right)\right)^2}{|x-y|^{1+\alpha}} dx dy \\
& = \frac{1}{(b-a)^{\alpha-1}} \int_0^1 \int_0^1 \frac{(\tilde f(x)- \tilde f(y))^2}{|x-y|^{1+\alpha}} dx dy .
\end{align*}
Since
$$
\int_0^1 \int_0^1 \frac{(\tilde f(x)- \tilde f(y))^2}{|x-y|^{1+\alpha}} dx dy \geq C^{(4)}_{\alpha} \tilde f^2 (1) = C^{(4)}_{\alpha} f^2 (b),
$$
by Part 1, the proof is complete.
\end{proof}
\noindent

\begin{proof}[Justification of Example \ref{ex:counter}]
We clearly have
\begin{align*}
\int_0^1 \int_0^1 \frac{(f_n(x)-f_n(y))^2}{|x-y|^{1+\alpha}}dxdy \leq \int_0^1 \int_0^{\frac{1}{2n}} \frac{(f_n(x)-f_n(y))^2}{|x-y|^{1+\alpha}}dxdy + \int_0^{\frac{1}{2n}} \int_0^1 \frac{(f_n(x)-f_n(y))^2}{|x-y|^{1+\alpha}}dxdy.
\end{align*}
By symmetry, the right hand side of the above inequality is equal to 
$$
2 \int_0^1 \int_0^{\frac{1}{2n}} \frac{(f_n(x)-f_n(y))^2}{|x-y|^{1+\alpha}}dxdy.
$$
Denote the last double integral by $J_n$. We have
$$
J_n \leq \sum_{l=0}^\infty \int_{\frac{2l}{n}}^{\frac{2l+2}{n}} \int_0^{\frac{1}{n}} \frac{(f_n(x)-f_n(y))^2}{|x-y|^{1+\alpha}}dxdy = \sum_{l=0}^\infty J_{n,l}.
$$
Recall that $f_n(x) = f(nx)$, where $f$ is a $C^{\infty}$-class function. Observe that for $x ,y \in [0,2]$ we have 
$$
|f_n(x) - f_n(y)| \leq \sup_{z \in [0,2]}|f_n^{\prime}(z)||x-y| = n \sup_{z \in [0,2]}|f^{\prime}(z)||x-y| \leq C n |x-y|.
$$
By this we obtain
$$
J_{n,0} = \int_0^{\frac{2}{n}} \int_0^{\frac{1}{n}} \frac{(f_n(x)-f_n(y))^2}{|x-y|^{1+\alpha}}dxdy \leq C^2 n^2 \int_0^{\frac{2}{n}} \int_{y-\frac{2}{n}}^{y+\frac{2}{n}} |x-y|^{1-\alpha} dxdy \leq C_{\alpha} n^{\alpha-1}.
$$
Similarly, for $l > 0$,
$$
J_{n,l} \leq \int_{\frac{2l}{n}}^{\frac{2l+2}{n}} \int_0^{\frac{1}{n}} \frac{1}{|x-y|^{1+\alpha}}dxdy \leq 2/n^2 \frac{n^{1+\alpha}}{(2l-1)^{1+\alpha}} \leq C_{\alpha} n^{\alpha-1} \left(\frac{1}{1+l}\right)^{1+\alpha}.
$$
Thus
$$
J_n \leq \sum_{l=0}^\infty J_{n,l} \leq C_{\alpha} n^{\alpha-1} \sum_{l=0}^\infty \left(\frac{1}{1+l}\right)^{1+\alpha} \leq C_{\alpha} n^{\alpha-1},
$$
and the proof is complete.
\end{proof}

\begin{proof}[Proof of Corollary \ref{cor:fwpi}]
Let $b_0 \in (a,b]$ be such that 
$$
f^2(b_0)g^2(b_0) = \max_{x \in (a,b]} f^2(x) g^2(x)
$$
We have  
\begin{align*}
\int_a^b \int_a^b \frac{(f(x)-f(y))^2}{|x-y|^{1+\alpha}} g(x) g(y) dx dy \geq g^2(b_0) \int_a^{b_0} \int_a^{b_0} \frac{(f(x)-f(y))^2}{|x-y|^{1+\alpha}}  dx dy, 
\end{align*}
which, by Theorem \ref{th:fpi}, is larger than 
\begin{align*}
\frac{C_{\alpha}^{(4)}}{(b_0-a)^{\alpha-1}} f^2(b_0) g^2(b_0) & \geq \frac{C_{\alpha}^{(4)}}{(b_0-a)^{\alpha-1}} \frac{1}{(b-a)} \int_a^b f^2(x) g^2(x) dx  = \frac{C_{\alpha}^{(4)}}{(b-a)^{\alpha}} \int_a^b f^2(x) g^2(x) dx.
\end{align*}
\end{proof}

\section{Spectral gap estimate}

\begin{proof}[Proof of Theorem \ref{th:sg}]
With no loss of generality we provide the arguments for the symmetric interval $(-a,a)$, $0<a<\infty$, only. Let $V \in \cV^{\alpha}((-a,a))$. Recall that the orthonormal basis $\left\{\varphi_n\right\}$ is chosen so that $\varphi_n$ are either symmetric or antisymmetric. If $\varphi_2$ is antisymmetric, then Theorem \ref{th:sg} follows from Theorem \ref{th:sga}. Assume now that $\varphi_2$ is symmetric. We directly deduce from Theorem \ref{th:gsp1} that the function $\varphi_2/ \varphi_1$ has a bounded derivative in each interval $[a_0,b_0]$, $-a<a_0<b_0<a$. Hence $\varphi_2/ \varphi_1$ is a Lipschitz function in each interval $[a_0,b_0] \subset (-a,a)$. Thus, by Proposition \ref{prop:var}, it is enough to estimate from below the double integral
$$
\int_{-a}^a \int_{-a}^a \frac{(f(x)-f(y))^2}{|x-y|^{1+\alpha}} \varphi_1(x) \varphi_1(y) dx dy, \ \ \ \ \ \ \ \ \text{with} \ \ \ \ \ \ f = \varphi_2 / \varphi_1.
$$ 
Note that $f$ is symmetric on $(-a,a)$, $f$ changes the sign in $(-a,a)$ and $\int_{-a}^a f^2(x) \varphi_1^2(x) dx = 1$. 

Let $a_0 = \min\left\{x \in [0,a): f(x) = 0\right\}$. Consider the following two cases.
\smallskip

\noindent
\emph{(Case 1)} Assume that 
$$
\int_{a_0}^a f^2(x) \varphi_1^2(x) dx \geq 1/4.
$$
We have
\begin{align*}
\int_{-a}^a \int_{-a}^a \frac{(f(x)-f(y))^2}{|x-y|^{1+\alpha}} \varphi_1(x) \varphi_1(y) dx dy 
& \geq \int_{a_0}^a \int_{a_0}^a  \frac{(f(x)-f(y))^2}{|x-y|^{1+\alpha}} \varphi_1(x) \varphi_1(y) dx dy \\
& \ \ \ + \int_{-a}^{-a_0} \int_{-a}^{-a_0}  \frac{(f(x)-f(y))^2}{|x-y|^{1+\alpha}} \varphi_1(x) \varphi_1(y) dx dy \\
& = 2 \int_{a_0}^a \int_{a_0}^a  \frac{(f(x)-f(y))^2}{|x-y|^{1+\alpha}} \varphi_1(x) \varphi_1(y) dx dy.
\end{align*}
Let now $b_0 \in [a_0,a)$ be such that $f^2(b_0)\varphi_1^2(b_0) = \max_{x \in (a_0,a)} f^2(x) \varphi_1^2(x)$. We have  
\begin{align*}
\int_{a_0}^a \int_{a_0}^a \frac{(f(x)-f(y))^2}{|x-y|^{1+\alpha}} \varphi_1(x) \varphi_1(y) dx dy \geq \varphi_1^2(b_0) \int_{a_0}^{b_0} \int_{a_0}^{b_0} \frac{(f(x)-f(y))^2}{|x-y|^{1+\alpha}}  dx dy, 
\end{align*}
which, by Theorem \ref{th:fpi}, is larger than 
\begin{align*}
\frac{C_{\alpha}^{(4)}}{(b_0-a_0)^{\alpha-1}} f^2(b_0) \varphi_1^2(b_0)  \geq \frac{C_{\alpha}^{(4)}}{(b_0-a_0)^{\alpha-1}} \frac{1}{(a-a_0)} \int_{a_0}^a f^2(x) \varphi_1^2(x) dx \geq \frac{1}{4}\frac{C_{\alpha}^{(4)}}{(a-a_0)^{\alpha}}.
\end{align*}
It follows that
$$
\int_{-a}^a \int_{-a}^a \frac{(f(x)-f(y))^2}{|x-y|^{1+\alpha}} \varphi_1(x) \varphi_1(y) dx dy \geq \frac{1}{2} \frac{C^{(4)}_{\alpha}}{(2a)^{\alpha}},
$$
which ends the proof in the first case.
\smallskip

\noindent
\emph{(Case 2)} Suppose now that
$$
\int_0^{a_0} f^2(x) \varphi_1^2(x) dx \geq 1/4.
$$
Notice that
\begin{equation}
\begin{split}
\label{eq:first}
\left(\int_{a_0}^a f(x) \varphi_1^2(x) dx \right)^2 \leq \int_{a_0}^a f^2(x) \varphi_1^2(x) dx \ \int_{a_0}^a \varphi_1^2(x)  dx  \leq \varphi_1^2(a_0)(a-a_0)\int_{a_0}^a f^2(x) \varphi_1^2(x) dx
\end{split}
\end{equation}
by Schwarz inequality and Theorem \ref{th:gsp}, and 
\begin{align}
\label{eq:second}
- \int_{a_0}^a f(x) \varphi_1^2(x) dx = \int_0^{a_0} f(x) \varphi_1^2(x) dx
\end{align}
by the fact that $f$ is symmetric and $\int_{-a}^a f(x) \varphi_1^2(x) dx = 0$. 
Observe that without loosing generality we may and do assume that $f \geq 0$ on $[0,a_0]$. Let $a^{*} \in [0,a_0)$ be such that $f(a^{*}) = \max_{x \in [0,a_0)} f(x)$. Note that $\int_0^a f^2(x) \varphi_1^2(x) dx = 1/2$. By \eqref{eq:first} and \eqref{eq:second}, we have
\begin{align*}
1/4 \geq \int_{a_0}^a f^2(x) \varphi_1^2(x) dx & \geq \frac{\left(\int_{a_0}^a f(x) \varphi_1^2(x) dx\right)^2}{\varphi_1^2(a_0)(a-a_0)} = \frac{\left(\int_0^{a_0} f(x) \varphi_1^2(x) dx\right)^2}{\varphi_1^2(a_0)(a-a_0)}\\
& = \frac{f^2(a^{*})\left(\int_0^{a_0} f(x) \varphi_1^2(x) dx\right)^2}{f^2(a^{*})\varphi_1^2(a_0)(a-a_0)} \geq \frac{\left(\int_0^{a_0} f^2(x) \varphi_1^2(x) dx\right)^2}{f^2(a^{*})\varphi_1^2(a_0)(a-a_0)},
\end{align*}
which implies that
\begin{align}
\label{eq:third}
f^2(a^{*}) \varphi_1^2(a_0) \geq 1/(4(a-a_0)).
\end{align}
We have 
\begin{align*}
\int_{-a}^a \int_{-a}^a \frac{(f(x)-f(y))^2}{|x-y|^{1+\alpha}} \varphi_1(x) \varphi_1(y) dx dy 
& \geq \int_0^{a_0} \int_0^{a_0}  \frac{(f(x)-f(y))^2}{|x-y|^{1+\alpha}} \varphi_1(x) \varphi_1(y) dx dy \\
& \ \ \ \ \ \ \ + \int_{-a_0}^0 \int_{-a_0}^0  \frac{(f(x)-f(y))^2}{|x-y|^{1+\alpha}} \varphi_1(x) \varphi_1(y) dx dy \\
& = 2 \int_0^{a_0} \int_0^{a_0}  \frac{(f(x)-f(y))^2}{|x-y|^{1+\alpha}} \varphi_1(x) \varphi_1(y) dx dy \\
& \geq 2 \varphi_1^2(a_0) \int_{a^{*}}^{a_0} \int_{a^{*}}^{a_0}  \frac{(f(x)-f(y))^2}{|x-y|^{1+\alpha}}.
\end{align*}
Now, using Theorem \ref{th:fpi} and \eqref{eq:third}, we obtain 
$$
\int_{-a}^a \int_{-a}^a \frac{(f(x)-f(y))^2}{|x-y|^{1+\alpha}} \varphi_1(x) \varphi_1(y) dx dy \geq 2 \varphi_1^2(a_0) \frac{C^{(4)}_{\alpha}}{(a_0-a^{*})^{\alpha-1}} f^2(a^{*}) \geq \frac{1}{2} \frac{C^{(4)}_{\alpha}}{(2a)^{\alpha}},
$$
which completes the proof.
\end{proof}
\smallskip
\noindent
\textbf{Acknowledgements.} I would like to thank Professor Tadeusz Kulczycki, my supervisor, for his help and guidance in investigating the theory and preparing this paper, which is part of my Ph.D. thesis.

\end{document}